\documentclass[10pt]{amsart}
\usepackage{pdflscape}

\usepackage{amscd,
verbatim,
amssymb,
amsmath,
stmaryrd,
mathrsfs
}
\usepackage[all]{xy}
\usepackage[colorlinks,linkcolor=blue,citecolor=blue,urlcolor=red]{hyperref}
\usepackage{bm}
\usepackage[ansinew,utf8x]{inputenc} 
\usepackage[ansinew]{inputenc} 
\usepackage[utf8x]{inputenc}
\usepackage{ulem} 
\usepackage{graphicx}
\usepackage{url}
\usepackage{bbm}
\usepackage[multiple]{footmisc}
\usepackage{imakeidx} 
\makeindex
\makeindex[name=not,title=Index of notation]

\usepackage{newunicodechar}
\newcommand{\PP}{\mathbb{P}}

\renewcommand{\AA}{\mathbb{A}}

\newcommand{\C}{\mathbb{C}}

\newcommand{\QQ}{\mathbb{Q}}

\newcommand{\Z}{\mathbb{Z}}
\newcommand{\ZZ}{\mathbb{Z}}
\newcommand{\sA}{\mathcal{A}}

\newcommand{\sB}{\mathcal{B}}

\newcommand{\sC}{\mathcal{C}}

\newcommand{\sF}{\mathcal{F}}

\newcommand{\sI}{\mathcal{I}}

\newcommand{\sL}{\mathcal{L}}

\newcommand{\OO}{\mathcal{O}}

\newcommand{\sS}{\mathcal{S}}
\newcommand{\sT}{\mathcal{T}}
\newcommand{\sU}{\mathcal{U}}
\newcommand{\sV}{\mathcal{V}}
\newcommand{\sW}{\mathcal{W}}
\newcommand{\sX}{\mathcal{X}}
\newcommand{\sY}{\mathcal{Y}}

\newcommand{\p}{\mathfrak{p}}

\newcommand{\Cor}{\operatorname{\mathbf{Cor}}}

\newcommand{\hA}{\ol{A}}
\newcommand{\hB}{\ol{B}}
\newcommand{\hC}{\ol{C}}

\newcommand{\hS}{\ol{S}}
\newcommand{\hT}{\ol{T}}
\newcommand{\hU}{\ol{U}}
\newcommand{\hV}{\ol{V}}
\newcommand{\hW}{\ol{W}}
\newcommand{\hX}{\ol{X}}
\newcommand{\hY}{\ol{Y}}

\newcommand{\iS}{S^\circ}
\newcommand{\iT}{T^\circ}

\newcommand{\iW}{W^\circ}
\newcommand{\iX}{X^\circ}
\newcommand{\iY}{Y^\circ}

\newcommand{\mA}{A^\infty}
\newcommand{\mB}{B^\infty}
\newcommand{\mC}{C^\infty}

\newcommand{\mS}{S^\infty}
\newcommand{\mT}{T^\infty}
\newcommand{\mU}{U^\infty}
\newcommand{\mV}{V^\infty}
\newcommand{\mW}{W^\infty}
\newcommand{\mX}{X^\infty}
\newcommand{\mY}{Y^\infty}

\newcommand{\ul}[1]{{\underline{#1}}}

\newcommand{\PSh}{{\operatorname{\mathbf{PSh}}}}
\newcommand{\Shv}{{\operatorname{\mathbf{Shv}}}}

\newcommand{\DM}{\operatorname{\mathbf{DM}}}

\newcommand{\ulMDM}{\operatorname{\mathbf{\underline{M}DM}}}

\newcommand{\ulMDA}{\operatorname{\mathbf{\underline{M}DA}}}

\newcommand{\Frac}{\operatorname{Frac}}

\newcommand{\Spec}{\operatorname{Spec}}

\newcommand{\uSpec}{\operatorname{\underline{Spec}}}

\newcommand{\Sm}{\operatorname{\mathbf{Sm}}}

\newcommand{\Ab}{\operatorname{\mathbf{Ab}}}

\newcommand{\tr}{{\operatorname{tr}}}

\newcommand{\eff}{{\operatorname{eff}}}

\newcommand{\fin}{{\operatorname{fin}}}

\newcommand{\red}{{\operatorname{red}}}

\newcommand{\Zar}{{\operatorname{Zar}}}
\newcommand{\Nis}{{\operatorname{Nis}}}
\newcommand{\et}{{\operatorname{\acute{e}t}}}

\newcommand{\fppf}{{\operatorname{fppf}}}

\newcommand{\nc}{{\operatorname{nc}}}
\newcommand{\sk}{{\operatorname{sk}}}
\newcommand{\cosk}{{\operatorname{cosk}}}

\renewcommand{\lim}{\operatornamewithlimits{\varprojlim}}

\newcommand{\colim}{\operatornamewithlimits{\varinjlim}}

\newcommand{\ol}{\overline}

\newcommand{\car}{\operatorname{char}}

\renewcommand{\epsilon}{\varepsilon}

\newcommand{\Bl}{{\mathbf{Bl}}}

\newcommand{\ulM}{\operatorname{\mathrm{\underline{M}}}}

\newcommand{\ulMZar}{\operatorname{\mathrm{\underline{M}Zar}}}
\newcommand{\ulMNis}{\operatorname{\mathrm{\underline{M}Nis}}}
\newcommand{\ulMet}{\operatorname{\mathrm{\underline{M}\acute{e}t}}}
\newcommand{\ulMfppf}{\operatorname{\mathrm{\underline{M}fppf}}}

\newcommand{\bcube}{{\ol{\square}}}

\newcommand{\ulPSm}{\operatorname{\mathbf{\underline{P}Sm}}}

\newcommand{\ulPSch}{\operatorname{\mathbf{\underline{P}Sch}}}
\newcommand{\ulPSCH}{\operatorname{\mathbf{\underline{P}SCH}}}
\newcommand{\ulMSm}{\operatorname{\mathbf{\underline{M}Sm}}}

\newcommand{\ulMSch}{\operatorname{\mathbf{\underline{M}Sch}}}
\newcommand{\ulMSCH}{\operatorname{\mathbf{\underline{M}SCH}}}

\newcommand{\ulMCor}{\operatorname{\mathbf{\underline{M}Cor}}}

\newcommand{\ulMO}{\ul{M}\OO}

\newcommand{\RMO}{\mathbf{\ul{M}\OO}}

\newcommand{\Comp}{\operatorname{\mathbf{Comp}}}

\newcounter{spec}
{\end{list}}%

\newtheorem{theorem}{Theorem}
\newtheorem{corollary}[theorem]{Corollary}

\newtheorem{lemma}{Lemma}[section]
\newtheorem{lemm}[lemma]{Lemma}

\newtheorem{theo}[lemma]{Theorem}

\newtheorem*{thm*}{Theorem}
\newtheorem{prop}[lemma]{Proposition}
\newtheorem{cor}[lemma]{Corollary}

\theoremstyle{definition}
\newtheorem{defi}[lemma]{Definition}

\newtheorem{nota}[lemma]{Notation}

\newtheorem{reco}[lemma]{Recollection}

\newtheorem{rema}[lemma]{Remark}
\newtheorem{exam}[lemma]{Example}
\theoremstyle{remark}

\numberwithin{equation}{section}

\def\Comp{\Comp^{\fin}}

\def\ulMSm{\operatorname{\mathbf{\ul{M}Sm}}}

\def\Comp{\operatorname{\mathbf{Comp}}}

\begin{document}

\title{Hodge cohomology with a ramification filtration, I}
\date{\today}
\author[S. Kelly]{Shane Kelly}
\address{
Graduate School of Mathematical Sciences
University of Tokyo
3-8-1 Komaba Meguro-ku
Tokyo 153-8914, Japan
}
\email{shanekelly@g.ecc.u-tokyo.ac.jp}
\author[H. Miyazaki]{Hiroyasu Miyazaki}
\address{NTT Corporation, NTT Institute for Fundamental Mathematics, 3-1 Morinosato-Wakamiya,Atsugi,Kanagawa 243-0198, Japan}
\email{hiroyasu.miyazaki.ah@hco.ntt.co.jp}
\thanks{
The first author was supported by JSPS KAKENHI Grant (19K14498). %
The second author is supported by JSPS KAKENHI Grant (21K13783).
}

\begin{abstract}
We consider a filtration on the cohomology of the structure sheaf indexed by (not necessarily reduced) divisors ``at infinity''. We show that the filtered pieces have transfers morphisms, fpqc descent, and are so called cube invariant. 

In the presence of resolution of singularities and weak factorisation they are invariant under blowup ``at infinity''. As such, they lead to a realisation functor from Kahn, Miyazaki, Saito and Yamazaki's category of motives with modulus over a characteristic zero base field.
\end{abstract}


\maketitle

\begin{center}
\today
\end{center}

\vspace{1cm}


\section{Introduction} \label{sec:intro}

In his celebrated work \cite{VoeTri}, Voevodsky constructed the triangulated category of mixed motives $\DM^{\eff}_k$ over a field $k$.
In the series of papers \cite{kmsy1}, \cite{kmsy2}, \cite{kmsy3}, Kahn, Miyazaki, Saito and Yamazaki define and study a triangulated category $\ulMDM^{\eff}_k$ which contains Voevodsky's category $\DM^{\eff}_k$ as a full subcategory.
One of their motivations is to obtain a motivic framework where one can study various non-$\AA^1$-invariant cohomology. An example of such a cohomology is the coherent cohomology $H^i_{\Zar}(-, \OO)$ of the structure sheaf $\OO$.
Indeed, $\OO$ is represented by $\AA^1$ which is contractible in $\DM^{\eff}_k$ by definition.

Somewhat surprisingly, it has be unknown for a long time whether the most obvious non-$\AA^1$-invariant cohomology theory $H^i_{\Zar}(-, \OO)$ is representable in $\ulMDM_k^\eff$ or not. In this paper we show that it is, at least over any field of characteristic $0$ (see Cor.~\ref{cor:main} below). 
As a consequence, we observe another fact, also surprisingly unknown for a long time, that $\ulMDM_k^\eff$ is strictly larger than $\DM_k^\eff$.

In fact, we will represent a suitable \emph{filtration} of $H^i_{\Zar}(-, \OO)$ graded by divisors ``at infinity'': 
for any choice of open immersion $X \subseteq \hX$ and invertible sheaf of ideals $\sI \subseteq \OO_{\hX}$ whose vanishing locus $\mX = \uSpec(\OO_{\hX}/\sI)$ satisfies $X = \hX \setminus \mX$, 
we can consider the image of the morphism\footnote{The strange use of $\sqrt{\sI}$ is motivated by the indexing of certain filtrations appearing in class field theory.} 
\[ H^q ((\hX,\mX),\ulMO):=H^q_{\Zar}(\hX, \sqrt{\sI} \otimes \sI^{\otimes -1}) \to H^q_{\Zar}(X, \OO_X). \]
In this way, we obtain a filtration on the cohomology of $X$ indexed by the multiplicity of effective Cartier divisors.
Moreover, one can prove that this filtration is exhaustive\footnote{All Nisnevich sheaves with transfers are canonically equipped with such so called ``motivic'' filtrations, and in the case this filtration is exhaustive the sheaf is called a \emph{reciprocity sheaf}, \cite{KSY16}, \cite{BKS22}, \cite{RS18}, \cite{SerreAGCF}.}:
\[
\colim_{n \geq 1} H^q ((\hX,n\cdot \mX),\ulMO) \xrightarrow{\sim} H^q_{\Zar}(X, \OO_X).
\]

As such, instead of the smooth varieties which generate $\DM^{\eff}_k$, the category $\ulMDM^{\eff}_k$ is generated by \emph{modulus pairs}. A modulus pair can be defined as a pair $\sX = (\hX, \mX)$ such that $\hX$ is a variety, $\mX \subseteq \hX$ is a closed subscheme, and $\iX := \hX \setminus \mX$ is smooth,\footnote{It is traditional to ask that the total space $\hX$ be normal, and $\mX$ be an effective Cartier divisor, however this can always be achieved by blowup and normalisation, two operations which we formally invert anyway.} %
In the same way that $M(X) \in \DM^{\eff}_k$ represents the cohomology of $X$, the object $M(\sX) \in \ulMDM^{\eff}_k$ represents a filtered piece of the cohomology of $\iX$, namely the cohomology with ramification bounded by $\mX$.

\ 

Our main theorem is the following.

\begin{theorem}[{Theorem~\ref{theo:MORealisation}}] \label{theo:main}
Let $k$ be a field
of characteristic zero. 
there exists an object $\RMO \in \ulMDM^\eff_k$ 
such that for any smooth variety $\hX$, any effective Cartier divisor $\mX$ with normal crossings support, and any $n \in \ZZ$ we have
\[
\hom_{\ulMDM^{\eff}_{k}} (M(\sX), \RMO[n]) \cong H_{\Zar}^{n}\left (\hX, \OO(\mX {-} |\mX|)\right ).
\]
where $\OO(\mX {-} |\mX|)$ is the line bundle associated to the divisor $\mX {-} |\mX|$.
\end{theorem}

As an immediate corollary of Theorem~\ref{theo:main}, we obtain 

\begin{corollary} \label{cor:main}
For any $X \in \Sm_k$, by taking $\sX = (X,\varnothing)$ in Theorem~\ref{theo:main}, we have a representation of the Hodge cohomology groups
\[
\hom_{\ulMDM^{\eff}_{k}} (M(X,\varnothing), \RMO[n]) \cong H_{\Zar}^{n} (X, \OO ).
\]
\end{corollary}

And of course:

\begin{corollary}
The canonical fully faithful inclusion 
\[ \DM^{\eff}_k \subseteq \ulMDM^{\eff}_k \]
 is not essentially surjective. 
\end{corollary}

The strategy can be summarised as follows. We define $\ulMO$ on $\ulPSm_k$ (recalled in \S \ref{sec:basic-def}) and show that it is a quasi-coherent {\'e}tale sheaf, \S \ref{sec:ulMO}, then show that its Nisnevich fibrant replacement is blowup invariant, \S \ref{sec:MObu}, cube invariant, \S \ref{sec:MOcube}, and has transfers, \S \ref{sec:MOtransfers}. 

We begin in \S \ref{sec:basic-def} with a recollection of the general theory, and in particular the construction of $\ulMDM_k^{\eff}$. 
In Appendix \ref{appendix-a}, we collect some definitions and facts about resolution of singularities and weak factorisations. 
In Appendix \ref{sec:MtauCoh}, we give a self-contained proof that $\ulMZar$ (resp. $\ulMet$) cohomology can be calculated as the colimit of Zariski cohomology (resp. \'etale cohomology) over abstract admissible blowups. In Appendix \ref{sec:cohCohCalc}, we make some basic computation of cohomologies on projective spaces generalising classical computations in SGA6 \cite{SGA6}.

In future work, this paper's results and techniques will be used to develop the analogue of Corollary~\ref{cor:main} for $H^p_{\Zar}(X, \Omega^q_X)$, as well as Hochschild homology with modulus satisfying an HKR isomorphism.

\emph{Related work.} In \cite{BPO22} a Hodge-type realization with log poles is constructed which should compare to the realisation constructed in this paper in case of reduced divisor or in case there is no divisor at infinity (that is, the case of tame ramification or the case of arbitrary ramification with no pole restriction). 

There is also work in progress by Marco D'Addezio, \cite{Dad23}, who studies the $\ulMZar$-sheafification of $(\hX, \mX) \mapsto \Gamma(\hX, \OO_{\hX})$ (although he writes ``simply marked schemes'' for modulus pairs and ``$v$-Zariski'' for $\ulMZar$). 

\subsection*{Acknowledgements}
We thank Dan Abramovich for clarifications about Definition~\ref{defi:NC}. 

We thank Ofer Gabber for bringing our attention to Example~\ref{exam:counterGabber}. 
Also, we thank Shuji Saito pointing out that the transfers version of our main theorem was possible; we were originally proving the result only for $\ulMDA_k^{\eff}$.

\section{Review of the general theory} \label{sec:basic-def}

In this section, we recall basic definitions concerning the category of modulus pairs, Rec.\ref{reco:1}, modulus topologies, Rec.\ref{reco2}, finite correspondences in the modulus setting, Rec.\ref{reco:Cor}, and the construction of $\ulMDM^{\eff}_k$, Rec.\ref{reco:MDM}.  One can find more details in many places: \cite{kmsy1, kmsy2, kmsy3}, \cite[\S 1]{nistopmod}, \cite[Chap.5, Chap.6]{KM21}. 

We fix a perfect base field $k$ with the case of interest being $\car(k) = 0$. We restrict our attention to modulus pairs over $k$ with smooth interior so as not to frighten the reader, but a large part of what we write holds over general bases, cf. \cite{KM21}.

\begin{reco}[Modulus pairs]\ \label{reco:1}
\begin{enumerate}
 \item A \emph{modulus pair}\footnote{We follow the terminology from  \cite[Def.1.1.1]{kmsy1}. There is a more general version studied in \cite{KM21} where $\hX$ is qc separated, and $\iX = \hX \setminus \mX$ is Noetherian.} over $k$ is a pair $$\sX = (\hX, \mX)$$ such that 
 \begin{enumerate}
  \item $\hX$ (called the \emph{total space}) is a separated $k$-scheme of finite type, 
  \item $\mX \subseteq \hX$ (called the \emph{modulus})  is an effective Cartier divisor, and 
  \item $\iX := \hX \setminus \mX$ (called the \emph{interior}) is smooth.
 \end{enumerate}
 
 \item An \emph{ambient $k$-morphism} $(\hX, \mX) \to (\hY, \mY)$ of modulus pairs is a $k$-morphism $f: \hX \to \hY$ of the underlying schemes, such that $\mX \geq f^*\mY$.\footnote{So for example, there is a tower of morphisms 
 \[ 
 \dots \to (\AA^1, (t^{n+1})) \to (\AA^1, (t^n)) \to \dots. 
 \]
 reflecting the fact that $k[t, t^{-1}] = \colim_n t^{1-n}k[t]$, cf.Example~\ref{exam:ulMOsncd}.}

 \item $\ulPSm_k$ is the category formed by modulus pairs over $k$, together with ambient $k$-morphisms.
 
 \item  $\ulMSm_k$ is the \emph{category of modulus pairs over $k$}. It is constructed by formally inverting the class $
\Sigma$ of \emph{abstract admissible blowups}:\footnote{So it is something like a ``global'' version of Raynaud's approach to rigid analytic spaces.} i.e., those ambient morphisms $f: (\hX, \mX) \to (\hY, \mY)$ such that 
\begin{enumerate}
 \item $\hX \to \hY$ is proper, 
 \item $\mX = f^*\mY$, and 
 \item $\iX \to \iY$ is an isomorphism. 
 \end{enumerate}
In symbols, 
\[ \ulMSm_k := \ulPSm_k[\Sigma^{-1}]. \]
It is shown in \cite[Prop.1.7.2]{kmsy1} (cf. also \cite[Prop.1.21]{KM21}) that $\Sigma$ admits a right calculus of fractions, so 
\[ \hom_{\ulMSm_k}(\sY, \sX) = \varinjlim_{\sY' \to \sY \in \Sigma} \hom_{\ulPSm_k}(\sY', \sX) \]
and the colimit is filtered; in fact its indexing category is a filtered poset. In particular every morphism in $\ulMSm_k$ can be written in the form $f \circ s^{-1}$ where $s \in \Sigma$ and $f$ is ambient. 

 \item \label{enum:PSmFib} The category $\ulPSm_k$\footnote{Of course if we insist on working in $\ulPSm_k$ we also need  $\iT \times_{\iS} \iX$ to be smooth over $k$, but in general the pullback along a minimal morphism basically always exists in the larger category of modulus pairs.} has categorical fibre products $\sY = \sT \times_{\sS} \sX$ in the case $f: \sT \to \sS$ is \emph{minimal}, i.e., in the case $\mT = j^*\mS$, \cite[Lem.1.32, Prop.1.33]{KM21}.
 \begin{enumerate}
  \item If $f$ is flat, $\hY = \hT \times_{\hS} \hX$ and $\mY = \mX|_{\hY}$. 
  \item If $f$ is an abstract admissible blowup, $\hY$ is the strict transform of $\hX$.
  \item For a general minimal $f$, the total space $\hY$ is the scheme theoretic closure of $\iT \times_{\iS} \iX$ in $\hT \times_{\hS} \hX$. 
 \end{enumerate}
 
 \item The category $\ulMSm_k$ admits all\footnote{Again, if we insist on working in $\ulMSm_k$ then we also need $\iT \times_{\iS} \iX$ to be smooth over $k$, but in general the large category of modulus pairs admits all fibre products.} categorical fibre products $\sY = \sT \times_{\sS} \sX$, \cite[Thm.1.40]{KM21}. The canonical functor $\ulPSm_k \to \ulMSm_k$ preserves the fibre products in item \eqref{enum:PSmFib}.
\end{enumerate}
\end{reco}

\begin{rema} \label{rema:sigmaSheaves}
By the universal property of localisation, the category $\PSh(\ulMSm_k)$ is canonically identified with the full subcategory of $\PSh(\ulPSm_k)$ consisting of those presheaves which send abstract admissible blowups to isomorphisms. Since abstract admissible blowups are categorical monomorphisms in $\ulPSm_k$, this is precisely the category of sheaves for the topology on $\ulPSm_k$ generated by abstract admissible blowups. In symbols, $\PSh(\ulMSm_k) = \Shv_\Sigma(\ulPSm_k)$, {cf.\cite[\S A.1]{KM21}}.
\end{rema}

\begin{reco}[Modulus topologies] \label{reco2}\ 
\begin{enumerate}
 \item The \emph{Zariski topology} on $\ulPSm_k$ is generated by families 
\begin{equation} \label{equa:zarcov}
\{f_i: (\hU_i, \mU_i) \to (\hX, \mX)\}_{i \in I}
\end{equation}
of minimal morphisms such that 
$\{\hU_i \to \hX_i\}_{i \in I}$ is a Zariski covering in the classical sense. Zariski coverings on $\ulPSm_k$ form a pretopology in the sense of \cite[Expos\'e II]{SGA41}. 

 \item The \emph{\underline{M}Zariski topology} or \emph{$\ulMZar$-topology} on $\ulMSm_k$ is generated by images of Zariski coverings under the localisation functor $\ulPSm_k \to \ulMSm_k$. Zariski coverings do not form a pretopology on $\ulMSm_k$. In general the coverings of the pretopology they generate consists of various iterated compositions of abstract admissible blowups, inverses of abstract admissible blowups, and Zariski coverings. However, such families can always be refined by one of the form 
\begin{equation} \label{equa:Mzarcov}
\{\sU_i \to \sY \to \sX\}_{i \in I}
\end{equation}
where $\{\sU_i \to \sY\}$ is a Zariski covering and $\sY \to \sX$ is an abstract admissible blowup, \cite[Cor.4.21]{KM21}.

 \item The $\ulMNis$, $\ulMet$, and \emph{$\ulMfppf$} topologies are defined in the analogous way to $\ulMZar$.
 There are also more exotic topologies considered in \cite{KM21} but we do not use them here.
\end{enumerate}
\end{reco}

\begin{rema} \label{rema:sigmaZarSheaves}
By Eq.\eqref{equa:Mzarcov} and Remark~\ref{rema:sigmaSheaves} the category $\Shv_{\ulMZar}(\ulMSm_k)$ is canonically identified with the full subcategory of $\Shv_{\Zar}(\ulPSm_k)$ consisting of those sheaves which send abstract admissible blowups to isomorphisms. In symbols, one could write $\Shv_{\ulMZar}(\ulMSm_k) = \Shv_{\Sigma,\Zar}(\ulPSm_k)$. 
\end{rema}

\begin{reco}[Finite correspondences] \label{reco:Cor} \ 
\begin{enumerate}
 \item Write $\Cor_k$ for Voevodsky's category of finite correspondences, \cite{VoeTri}. Objects are smooth $k$-schemes and $\hom_{\Cor_k}(X, Y)$ is the free abelian group 
\[ \ZZ \left \{ Z\ \middle |\  \begin{array}{c} Z \subseteq X \times Y \textrm{ is an integral closed subscheme, and } \\
Z \to X \textrm{ is finite and dominates an irreducible component of } X \end{array} \right \}. \] 
There is a canonical functor $\Sm_k \to \Cor_k$ which sends a morphism $f: X \to Y$ to the graph $[f] := \sum \delta(X_i)$ where $X_i \subseteq X$ are the irreducible components of $X$ and $\delta: X \to X \times Y$ is the graph morphism $x \mapsto (x, f(x))$. %
If $f: X' \to X$ is \'etale and $\sum n_i Z_i \in \hom_{\Cor_k}(X, Y)$ any correspondence, then $\alpha \circ [f] = \sum n_i \sum Z'_{ij}$ where $Z'_{ij}$ are the irreducible components of $X' \times_X Z_i$.

 \item \label{reco:Cor:Omega} The structure presheaf $\OO$ on $\Sm_k$ which send $X$ to $\Gamma(X,\OO_X)$ has a structure of transfers in the sense that there exists $\OO_{\tr}: \Cor_k^{op} \to \Ab$ such that $\OO_{\tr}|_{\Sm_k} = \OO$, \cite{SVSing}.

 \item We write $\ulMCor_k$ for the category of modulus correspondences.\footnote{Cf.\cite{kmsy1, kmsy2, kmsy3} when the base is a field, or \cite{KM21} for general bases, where more general means modulus pairs $\sS$ such that $\hS$ is quasi-compact separated and such that $\iS$ is Noetherian.} %
Objects are the same as $\ulMSm_k$ and morphism groups can be defined as the intersections\footnote{If the reader is scared of correspondences for non-smooth schemes, we point out that the formula still if we replace $\colim_{\sW \to \sX} \hom_{\Cor_k}(\iW,\iY)$ with $\colim_{\sW \to \sX} \colim_{U \subseteq \iW} \hom_{\Cor_k}(U,\iY)$ where the $U \subseteq \iW$ are regular dense open subschemes.}
\[ \xymatrix@R=6pt{
\hom_{\ulMCor_k}(\sX,\sY) \ar@{}[r]|-\subseteq \ar@{}[dd]|-{\cap|} &  \hom_{\Cor_k}(\iX,\iY) \ar@{}[d]|-{\cap|} \\
& \colim_{\sW \to \sX} \hom_{\Cor_k}(\iW,\iY) \ar@{=}[d] \\
\colim_{\sW \to \sX} \ZZ \hom_{{\ulMSm_k}}(\sW,\sY) \ar@{}[r]|-\subseteq & \colim_{\sW \to \sX} \ZZ \hom_{{\Sm_k}}(\iW,\iY)
} \]
where the colimit is over ambient minimal morphisms $\sW \to \sX$ such that $\hW \to \hX$ is proper surjective, $\iW \to \iX$ is finite, and $\hW$ is integral, cf.\cite[Cor.4.21, Prop.4.37, Cor.5.31]{KM21}.\footnote{In \cite[Cor.5.31]{KM21}, qfh-sheafifications are used instead of the colimits, but one sees that this is the same using arguments such as \cite[Prop.3.3.1]{Voe96}.} 
More explicitly, for $\iX$ integral, a correspondence $\alpha: \iX \to \iY$ belongs to $\ulMCor_k (\sX,\sY)$ if and only if there exists a proper surjective morphism $\hW \to \hX$ with $\iW \to \iX$ finite and $\hW$ integral, and a finite sum $\sum n_i f_i$ of morphisms\footnote{The  category $\ulMSch_k$ of not necessarily smooth modulus pairs is defined in the analogous way: objects are pairs $(\hX, \mX)$ with $\hX$ separated and finite type over $k$, and the modulus $\mX$ is an effective Cartier divisor. Ambient morphisms $(\hX, \mX) \to (\hY, \mY)$ are those $k$-morphisms $\hX \to \hY$ such that $\mX \geq \mY|_{\hX}$. Admissible blowups are morphisms $(\hX, \mX) \to (\hY, \mY)$ such that $\hX \to \hY$ is proper, $\mX = \mY|_{\hX}$ and $\iX = \iY$. The category $\ulMSch_k$ is obtained from $\ulPSch_k$ by inverting all admissible blowups.} $f_i: \sW \to \sY$ such that $\alpha = \sum  n_i f_i^\circ: \iW \to \iY$, where $\sW = (\hW, \hW \times_{\hX} \mX)$.

\[ \xymatrix{
\iW \ar[r] \ar[d] \ar@/^12pt/[rr]^{\sum n_i f_i^\circ} & \iX \ar[r]^\alpha \ar[d] & \iY \ar[d] \\
\sW \ar[r] \ar@/_12pt/[rr]_{\sum n_i f_i} & \sX \ar@{-->}[r] & \sY
} \]

Taking graphs induces a covariant functor $\ulMSm_k \to \ulMCor_k$. 
\end{enumerate}
\end{reco}

\begin{reco} \label{reco:MDM}\ 
\begin{enumerate}
 \item An additive presheaf on $\ulMCor_k$ is called a \emph{modulus presheaf with transfers}. 
Let $\tau \in \{\Zar, \Nis, \et\}$. 
A modulus presheaf with transfers is called a \emph{$\tau$-sheaf with transfers} if for any modulus pair $\sX$, the presheaf $(F|_{\ulMSm_k})_{\sX}$ on the small site $\hX_{\tau}$ is a $\tau$-sheaf, where   $F|_{\ulMSm_k}$ denotes the restriction via the graph functor $\ulMSm_k \to \ulMCor_k$.

One can prove that $\Shv_{\ulMNis}(\ulMCor_k)$ is a Grothendieck abelian category using the usual methods, \cite[Cor.6.8]{KM21}, i.e., by showing that the forgetful functor admits a left adjoint
\[ a_{\tr}: \PSh(\ulMCor_k) \to \Shv_{\ulMNis}(\ulMCor_k) \]
which becomes sheafification when restricted to $\ulMSm_k$. 
Moreover, by classical arguments with input from \cite{KM21},%
\footnote{Use the proof of \cite[Prop.3.1.7]{VoeTri} with \cite[Prop.6.2, Cor.6.8]{KM21} inserted at the appropriate places.} %
for any $F \in \Shv_{\ulMNis}(\ulMCor_k)$ we have
\begin{equation} \label{equa:ExtCoh}
Ext^\bullet(\ZZ_{\tr}(\sX), F) \cong H_{\ulMNis}^\bullet(\sX, F)
\end{equation}
where $\ZZ_{\tr}(\sX) = \hom_{\ulMCor_k}(-, \sX)$.

 \item In analogy with $\DM_k^{\eff}$ from \cite{VoeTri} the category $\ulMDM_k^{\eff}$ is defined to be the Verdier quotient 
\[ \ulMDM_k^{\eff} = 
\frac{D(\Shv_{\ulMNis}(\ulMCor_k))}{\biggl \langle \ZZ_{\tr}(\sX\otimes \bcube) \to \ZZ_{\tr}(\sX) \biggm | \sX \in \ulMCor_k \biggr \rangle}
 \]
by the two term complexes $[\ZZ_{\tr}(\sX \otimes \bcube) \to \ZZ_{\tr}(\sX)]$ where $\sX \otimes \bcube = (\hX \times \PP^1, \mX \times \PP^1 + \hX \times \{\infty\})$. 
Since the generators $\ZZ_{\tr}(\sX)$ are compact, \cite[Thm.4.47]{KM21}, the localisation functor admits a right adjoint, and $\ulMDM_k^{\eff}$ can be identified with the full subcategory 
\begin{align}
&\ulMDM_k^{\eff} \cong \label{eq:MDMsub} \\
\biggl \{ K \in &D(\Shv_{\ulMNis}(\ulMCor_k))\ |\ \mathbb{H}_{\ulMNis}^\bullet(\sX, K) \cong \mathbb{H}_{\ulMNis}^\bullet(\sX \otimes \bcube, K) \biggr \} \nonumber
\end{align}
of objects with cube invariant hypercohomology, cf Eq.\eqref{equa:ExtCoh}.
\end{enumerate}
\end{reco}

\section{The presheaf $\protect\ulMO$} \label{sec:ulMO}

\begin{defi} \label{defi:ulMORing}
If $A$ is any ring and $f \in A$ a nonzero divisor we write 
\[ \ulMO (A, f) := \{ a/f \in A[f^{-1}] : a \in \sqrt{(f)} \subseteq A \}. \]
\end{defi}

Of course, $A \subseteq \ulMO (A, f)$ with equality if $f$ is invertible, and on the other side, $\bigcup_{n \geq 0} \ulMO(A, f^n) = A[f^{-1}]$.

\begin{exam}[Cf.the proof of Lemma~\ref{lemm:flatMO}] \label{exam:ulMOsncd}
If $A = \QQ[x_1, \dots, x_n]$ and $f = x_1^{r_1}\dots x_i^{r_i}$ with $r_1, \dots, r_i > 0$ then 
\[ \ulMO (A, f) = \tfrac{1}{x_1^{r_1{-}1}\dots x_i^{r_i{-}1}} \QQ[x_1, \dots, x_n]. \]
More generally, if $A$ is a UFD, 
$f_1, \dots, f_i \in A$ are pair-wise distinct irreducible elements, and $f = f_1^{r_1} \dots f_i^{r_i}$ with $r_1, \dots, r_i > 0$ then 
\[ \ulMO (A, f) = \tfrac{1}{f_1^{r_1{-}1}\dots f_i^{r_i{-}1}}A. \]
In particular, $\ulMO (A, f)$ is free of rank one in this case.
\end{exam}

\begin{lemm} \label{lemm:affineEtQC}
Suppose $\phi: A \to B$ is a homomorphism of rings equipped with nonzero divisors $f, g$ respectively, such that $\phi(f)$ divides $g$. Then $A[f^{-1}] \to B[g^{-1}]$ induces a morphism of submodules
\[ \ulMO(A, f) \to \ulMO(B, g). \]
\end{lemm}

\begin{proof}
If $g = \phi(f)$, the question is whether 
$\frac{1}{\phi(f)}\phi(\sqrt{fA}) 
\subseteq 
\frac{1}{\phi(f)}\sqrt{\phi(f)B}$ 
or equivalently, whether 
$\phi(\sqrt{fA}) 
\subseteq 
\sqrt{\phi(f)B}$, which is clear since $a^n = fh$ for some $h$ implies $\phi(a)^n =\phi(f)h'$ for some $h'$. 
If $A = B$, so $g = fh$ for some $h$, the question is whether 
$\frac{1}{f}\sqrt{fA} 
\subseteq 
\frac{1}{fh}\sqrt{fhA}$ 
or equivalently, whether
$h\sqrt{fA} 
\subseteq 
\sqrt{fhA}$ which is also clear since if $a^n = fj$ for some $j$ and $n \geq 1$ then $(ha)^n = fhj'$ for some $j'$.
We get the general case by factoring the given morphism as $(A, f) \to (B, \phi(f)) \to (B, g)$.
\end{proof}

\begin{prop} \label{prop:fpqc}
Suppose that $A$ is a ring, $f \in A$ is a nonzero divisor and $A \to B$ a faithfully flat morphism (in particular, the images of $f$ in $B$ and $B \otimes_A B$ are again nonzero divisors). Then
\[ 0 \to \ulMO(A, f) \to \ulMO(B, f) \to \ulMO(B \otimes_A B, f) \]
is exact, where we write $f$ also for the images in $B$ and $B\otimes_A B$ to lighten the notation.
\end{prop}

\begin{proof}
We are studying the diagram:
\[ \xymatrix@R=12pt{
0 \ar[r] & A[\frac{1}{f}] \ar[r] & B[\frac{1}{f}] \ar[r] & (B {\otimes}_A B)[\frac{1}{f}] \\
0 \ar[r] & \ulMO(A, f) \ar[u]_{\cup|} \ar[r] & \ulMO(B, f) \ar[u]_{\cup|} \ar[r] & \ulMO(B {\otimes}_A B, f) \ar[u]_{\cup|}  \\
0 \ar[r] & A \ar[u]_{\cup|} \ar[r] & B \ar[u]_{\cup|} \ar[r] & B {\otimes}_A B \ar[u]_{\cup|} 
} \]
The map $\ulMO(A,f) \to \ulMO(B,f)$ is injective because it is induced by the monomorphism $A[\frac{1}{f}] \to B[\frac{1}{f}]$. Suppose that $a \in \ker(\ulMO(B,f) \to \ulMO(B{\otimes}_A B,f))$. By exactness of the top row,  
the element $a$ is in the subgroup $A[\frac{1}{f}] \subseteq B[\frac{1}{f}]$. We want to show that $a \in \ulMO(A,f)$. That is, we want to show that there is $n \geq 0$ such that $fa, (fa)^na \in A$. Since $a \in \ulMO(B, f)$, there is $n\geq0$ such that $fa, (fa)^na \in B$. But $B{\otimes}_AB \to (B{\otimes}_AB)[\frac{1}{f}]$ is injective, so the cocycle condition for $a$ implies that $fa$ and $(fa)^na$ are cocycles, so we see that $fa$ and $(fa)^na$ are in the subgroup $A \subseteq B$ as desired.
\end{proof}

\begin{lemm} \label{lemm:etaleAffineCoh}
Let $\phi: A \to B$ be an {\'e}tale homomorphism and $f \in A$ a nonzero divisor. Then the canonical isomorphism $A[f^{-1}] \otimes_A B \stackrel{\sim}{\to} B[\phi(f)^{-1}]$ induces an isomorphism of submodules
\[ \ulMO(A, f) \otimes_A B \stackrel{\sim}{\to} \ulMO(B, \phi(f)). \]
\end{lemm}

\begin{exam} \label{exam:notFpqcCoh}
Even though $\ulMO$ is an fpqc sheaf, Prop.~\ref{prop:fpqc}, the statement of Lemma~\ref{lemm:etaleAffineCoh} does not generalise to flat morphisms. For $(\C[[t^2]], t^4) \to (\C[[t]], t^4)$ the morphism $B \otimes_A \ulMO(A, f) \to \ulMO(B, f)$ is $t^{-2}\C[[t]] \to t^{-3}\C[[t]]$. Of course, this does not preclude the possibility that the comparisons $H_{\Zar}^n(\hX, \ulMO) \to H_{\fppf}^n(\hX, \ulMO)$ be isomorphisms.
\end{exam}

\begin{proof}
First note that we have $\phi(\sqrt{fA})B = \sqrt{\phi(f)}B$ because $\phi$ is {\'e}tale: Indeed, to prove this claim it suffices to show that $\psi: (A/fA)_{\red} \otimes_A B \to (B /\phi(f)B)_{\red}$ is an isomorphism. By \cite[033B]{stacks-project} the ring $(A/fA)_{\red} \otimes_A B$ is already reduced because it is {\'e}tale over the reduced ring $(A/fA)_{\red}$, so it suffices that $\Spec(\psi)$ be a surjective closed immersion. This happens because $(A/fA) \otimes_A B \to B /\phi(f)B$ is an isomorphism, and $\Spec((A/fA)_{\red}) \to \Spec(A/fA)$ and $\Spec(B /\phi(f)B)_{\red} \to \Spec(B/\phi(f)B$ are surjective closed immersions.

Now since $A \to B$ is flat, $\sqrt{fA} \otimes_A B \stackrel{\sim}{\to} \phi(\sqrt{fA})B$ so we have $\sqrt{fA} \otimes_A B  \stackrel{\sim}{\to} \sqrt{\phi(f)}B$, and therefore $\frac{1}{f} \sqrt{fA} \otimes_A B \stackrel{\sim}{\to} \frac{1}{\phi(f)}\sqrt{\phi(f)}B$, considered as submodules of $A[f^{-1}] \otimes_A B \stackrel{\sim}{\to} B[\phi(f)^{-1}]$.
\end{proof}

\begin{theo} \label{theo:MOdefi}
There is a unique $\fppf$-sheaf $\ulMO$ on $\ulPSm_k$ such that for affine modulus pairs with principal modulus $(\Spec(A), (f))$ we have $\ulMO(\Spec(A), (f)) = \ulMO(A, f)$. Furthermore, this is quasi-coherent as an {\'e}tale sheaf. In particular, its Zariski, Nisnevich, and {\'e}tale cohomologies agree, and vanish for affines.
\end{theo}

\begin{proof}
Define $\ulMO$ to be the Zariski-sheafification of the right Kan extension from affines to all of $\ulPSm_k$. On affines, the Zariski sheafification is the same as the right Kan extension because of Proposition~\ref{prop:fpqc}. This Zariski sheaf is automatically an {fpqc}-sheaf by Proposition~\ref{prop:fpqc} and \cite[03O1]{stacks-project}.

Lemma~\ref{lemm:affineEtQC} says that this \'etale sheaf is quasi-coherent in the sense of \cite[Exa.II.1.2(d), Cor.II.1.6]{Mil80}, so the cohomologies agree, \cite[Rem.III.3.8]{Mil80}, and vanish for affines by \cite[01XB]{stacks-project}.
\end{proof}

\section{Blow-up invariance of $R\Gamma(-,\protect\ulMO)$} \label{sec:MObu}

In the previous section, we have constructed a Zariski (in fact fpqc) sheaf $\ulMO$ of abelian groups on $\ulPSm_k$ which is quasi-coherent as an \'etale sheaf. 
Our next goal is to prove that $\ulMO$ and its cohomology presheaves are invariant under suitable blow-ups. For global sections $\ulMO$ we need a normality assumption, Prop.\ref{prop:hdescent}. For the cohomology we assume normal crossings, Prop.\ref{prop:MOBUinv}.

To begin with we characterise of elements of $\ulMO(A,f)$ in Lemma~\ref{lemm:memCritMO}, and show that if $A$ is normal, $\ulMO(A,f)$ is normal in the sense of Barth,\footnote{Cf.\cite[pg.128]{Bar77}.} Lem.\ref{lemm:intersectHeightOne}.  

\begin{lemm} \label{lemm:memCritMO}
Let $A$ be a ring, $f \in A$ a nonzero divisor, and $a \in A[f^{-1}]$. Then $a \in \ulMO(A,f)$ if and only if for some $n \geq 0$ both $fa, (fa)^na \in A$.
\end{lemm}

\begin{proof}
Certainly, if $a \in \ulMO(A,f)$, then $af \in \sqrt{(f)} \subseteq A$ so $(af)^{n+1} = bf$ for some $n \geq 0, b \in A$ so $(af)^na \in A$ for some $n \geq 0$. On the other hand, if $af, (af)^na \in A$ for some $n \geq 0$, then $af \in \sqrt{(f)}$ so $a \in \ulMO(A, f)$.
\end{proof}

\begin{lemm} \label{lemm:intersectHeightOne}
Let $A$ be a Noetherian normal domain and $f \in A$ a nonzero divisor. Then there is an equality
\[ \ulMO(A, f) = \bigcap_{\textrm{height } \p = 1} \ulMO(A_\p, f) \]
of sub-$A$-modules of $\Frac(A)$.
\end{lemm}

\begin{proof}
The inclusion 
$\ulMO(A, f) \subseteq \bigcap_{\textrm{height } \p = 1} \ulMO(A_\p, f)$
comes from Lemma~\ref{lemm:memCritMO}. 
Suppose we have an element $a$ on the right. Then for all height one primes $\p$, we have $a \in A_\p[f^{-1}] = A[f^{-1}]_\p$ and  there exists $n_{\p}$ such that $fa, (fa)^{n_{\p}} a \in A_\p$ by Lemma~\ref{lemm:memCritMO}.
Since $\in A$ is an open condition, for each $\p$, there exists an open neighborhood $U_{\p}$ of $\p$ on which $fa$ and $(fa)^{n_{\p}}a$ are still regular functions.
Since $\Spec A$ is quasi-compact, there exists a finite family $\p_1,\dots,\p_m$ such that  $\Spec A = \cup_{i=1}^m U_{\p_i}$. 
Set $n:=\max (n_{\p_1},\dots,n_{\p_m})$.
Then we have $fa \in A$ and $(fa)^n \in A$, and hence belongs to the left hand side, as desired.
\end{proof}

\begin{prop} \label{prop:hdescent}
Suppose that $\sY \to \sX$ is a morphism in $\ulPSm_k$ such that $\mY = \hY \times_{\hX} \mX$ and $\hY \to \hX$ is a proper, surjective morphism with normal target. 
Then the square 
\[ \xymatrix{
\ulMO(\sY) \ar[r] & \OO(\iY) \\
\ulMO(\sX) \ar[r]  \ar[u] & \OO(\iX) \ar[u]
} \] 
is Cartesian. In particular, there is a unique presheaf on $\ulMSm_k$ whose restriction to $\ulPSm_k$ agrees with $\ulMO$ on integrally closed modulus pairs.
\end{prop}

\begin{proof}
By definition $\ulMO$ is a Zariski sheaf, so we can assume $\hX$ is affine and $\mX$ has a global generator, say $\mX = (f)$. Suppose $a \in \OO(\iX)$ is a section whose image in $\OO(\iY)$ lies in $\ulMO(\sY)$. For all points $y \in \hY$, by Lemma~\ref{lemm:memCritMO} there is some $n_y \geq 0$ for which $fa$ and $(fa)^{n_y}a$ are in $\OO_{\hY, y}$. Since $\hY$ is quasi-compact, there is some $n$ which works for all $y$. So in fact, $fa$ and $(fa)^{n}a$ are in $\OO(\hY)$. Applying Lemma~\ref{lemm:memCritMO} again, it suffices to show that $fa$ and $(fa)^{n}a$ are in $\OO(\hX)$. Since $\hX$ is normal, it suffices to show that they are in the dvrs $\OO_{\hX,x}$ for points $x \in \hX$ of codimension one, Lem.\ref{lemm:intersectHeightOne}. Chose any extension of the valuation of $\OO_{\hX,x}$ to $L$ and let $\OO_L$ be the corresponding valuation ring. Since $\hY \to \hX$ is proper, the morphism $\Spec(\OO_L) \to \Spec(\OO_{\hX,x}) \to \hX$ factors as $\Spec(\OO_L) \to \hY$, so we find that the images of $fa$ and $(fa)^{n}a$ in $L$ are in fact in $\OO_L$. That is, they have value $\geq 0$. Hence, they are in $\OO_{\hX,x}$.
\end{proof}

\begin{rema}
Since $\OO = \hom(-, \AA^1)$ is an $h$-sheaf on the category of normal schemes, \cite[Prop.3.2.10]{Voe96}, it follows from Prop.\ref{prop:fpqc} and Prop.\ref{prop:hdescent} that $\ulMO$ is an $h$-sheaf on normal modulus pairs, but we will not need this.
\end{rema}

\begin{nota} \label{nota:smallsiteMO}
In the following proposition and proof we use $\ulMO_{\hX} = \ulMO|_{\hX_{\Zar}}$ or $\ulMO_{\hX_{\et}} = \ulMO|_{\hX_{\et}}$ for the restrictions to the small Zariski, resp. \'etale sites.
\end{nota}

\begin{prop}[Blow-up invariance of $\ulMO$ and its cohomologies] \label{prop:MOBUinv}
Take $\sX \in \ulPSm_k$ normal crossings and suppose that $Z \subseteq \hX$ is a closed subscheme that has normal crossings with $\mX$ (see Def.\ref{defi:NC} for the terminology). 
Let $f: \hY \to \hX$ be the blowup with centre $Z$, and $\mY = f^* \mX$. Then the canonical morphism in the derived category of quasi-coherent $\OO_{\hX}$-modules
\[ \ulMO_{\hX} \to Rf_*\ulMO_{\hY}\]
is an isomorphism. 
\end{prop}

\begin{proof}
Since the \'etale and Zariski cohomologies of $\ulMO$ agree, Thm.~\ref{theo:MOdefi}, it suffices to show that the morphism $\alpha : \ulMO_{\hX_{\et}} \to Rf_*\ulMO_{\hY_{\et}}$ is an isomorphism, where now $Rf_*$ is the direct image between the small \'etale sites. Since this question is {\'e}tale local, it suffices to find for each point $x \in \hX$ an {\'e}tale morphism $\hU \to \hX$ whose image contains $x$, and such that $\ulMO_{\hX_{\et}}|_{\hU} \to Rf_*\ulMO_{\hY_{\et}}|_{\hU}$ is an isomorphism.

Fix a point $x \in \hX$. By the definition of normal crossings with $\mX$ (Def.~\ref{defi:NC}), there exists a diagram \[\hX \xleftarrow{p} \hU \xrightarrow{q} \AA^n = \Spec k[t_1,\dots,t_n]\] of \'etale morphisms such that $x \in p(\hU)$, $p^*\mX = q^*H$ and $Z \times_{\hX} \hU = Z_0 \times_{\AA^n} \hU$, where $H=\{\prod_{a \in A} t_a^{r_a}=0\}$ and $Z_0=\{t_b=0, \forall b \in B\}$ for some $r_a > 0$ and $A,B\subset \{1,\dots,n\}$. Therefore, replacing $f$ by $f|_{\hU}$, we may assume that there exists an \'etale morphism $q : \hX \to \AA^n$ such that $X^\infty = q^* H$ and $Z=Z_0 \times_{\AA^n} \hX$ with $H = \{t_1^{m_1}\cdots t_r^{m_r}=0\}$ and $Z_0=\{t_b=0, \forall b \in B\}$ as above.

Since $\hX \to \AA^n$ is \'etale (hence flat), we obtain a cartesian diagram
\[\xymatrix{
\hY \ar[r]^{q'} \ar[d]_f \ar@{}[rd]|\square & \hY_0 \ar[d]^{f_0} \\
\hX \ar[r]_q & \AA^n
}\]
where $f_0$ is the blow-up of $\AA^n$ at $Z_0$ and $q'$ is the morphism induced by the universal property of blow-up.
Suppose that we know that the assertion of Prop.~\ref{prop:MOBUinv} holds for $f_0$. That is, suppose that,
setting $\sA:=(\AA^n,H)$ and $\sY_0:=(\hY_0,f_0^*H)$, we have an isomorphism $\ulMO_{\hA} \cong Rf_{0*} \ulMO_{\hY_0}$. 
By applying $q^* = Rq^*$ to this isomorphism and by using the flat base change $q^* Rf_{0*} \cong Rf_* q^{\prime*}$ \cite[02KH]{stacks-project}, we obtain 
\[
q^* \ulMO_{\hA} \cong q^* Rf_{0*} \ulMO_{\hY_0} \cong 
Rf_* q^{\prime*} \ulMO_{\hY_0}.
\]
On the other hand, we have $q^* \ulMO_{\hA} = \ulMO_{\hX}$ and $q^{\prime*} \ulMO_{\hY_0} = \ulMO_{\hY}$ by quasi-coherence since $q$ and $q'$ are \'etale by construction. Assembling all these isomorphisms leads to the isomorphism
$\ulMO_{\hX} \cong Rf_*\ulMO_{\hY}$, and applying Thm.~\ref{theo:MOdefi} again gets us to the desired isomorphism $\ulMO_{\hX_{\et}} \cong Rf_*\ulMO_{\hY_{\et}}$.

So it now suffices to prove Proposition~\ref{prop:MOBUinv} in the special case $f_0: \Bl_{\AA^n} Z_0 \to \AA^n$. This is done by direct calculation in Proposition~\ref{prop:BasicBlowupInvariant} below, using Lemma~\ref{lem:MOt} 
to reduce to the case $Z_0 = \{0\}$.
\end{proof}

\begin{lemma}\label{lem:MOt}
For any ring $A$ and for any non-zero divisor $f \in A$, we have 
\[
\ulMO (A[t],f) \cong \ulMO (A,f)[t].
\]
\end{lemma}

\begin{proof}
Let $I=(f)$. By unpacking the definition of $\ulMO$, we are immediately reduced to showing the equality $\sqrt{I[t]} = \sqrt{I}[t]$ of ideals of $A[t]$.
We first prove $\sqrt{I[t]} \subset \sqrt{I}[t]$. Since $\sqrt{I[t]}$ is the smallest radical ideal containing $I[t]$, it suffices to show that $\sqrt{I}[t]$ is a radical ideal. This means by definition that the quotient ring $A[t]/\sqrt{I}[t]$ is reduced. But this follows from $A[t]/\sqrt{I}[t] \cong A/\sqrt{I} [t]$.
To see the opposite inclusion $\sqrt{I[t]} \supset \sqrt{I}[t]$, take any polynomial $\sum_i a_i t^i$ of degree $d$ with $a_i \in \sqrt{I}$. Then for each $i$, there exists $n_i >0$ such that $a_i^{n_i} \in I$. Take an integer $N > (d+1)\max \{n_i\}$. Then one checks that $(\sum_i a_i t^i)^N \in I[t]$, and hence $\sum_i a_i t^i \in \sqrt{I[t]}$.
\end{proof}

\begin{prop} \label{prop:BasicBlowupInvariant}
Consider the blowup $f: \hB = \Bl_{\AA^n}\AA^d \to \hA$ of affine space $\hA = \Spec(k[t_1, \dots, t_n])$ along a sub-affine space $\AA^d \subset \AA^n$, equip $\hA$ with the divisor $\mA = t_1^{r_1} \dots t_i^{r_i}$ with $r_j, i \geq 1$ and equip $\hB$ with the pullback $\mB$ to obtain an abstract admissible blowup $\sB = (\hB, \mB) \to (\hA, \mA) = \sA$. 
Then we have
\[ \ulMO_{\hA} \cong Rf_*\ulMO_{\hB} \]
where $\ulMO_{\hA} = \ulMO|_{\hA_{\Zar}}$ means restriction to the small Zariski site (and similar for $\ulMO_{\hB}$).
\end{prop}

\begin{proof}
First we reduce the assertion to the case that $d=0$. 
Note that $\AA^n \cong \AA^{n-d} \times \AA^d$ and $\AA^d \cong \{0\} \times \AA^d$. These identifications induce an isomorphism 
\[
\Bl_{\AA^n} \AA^d \cong (\Bl_{\AA^{n-d}} \{0\}) \times \AA^d
\]
since strict transform along a flat morphism is a pullback \cite[0805]{stacks-project}. Combining this with Lemma~\ref{lem:MOt}, we are reduced to the case $\AA^d=\{0\}$. 

Now assume $\AA^d=\{0\}$.
Observe that the pullback $f^*|\mA|$ of the support $|\mA|$ is $f^*|\mA| = |\mB| + (i{-}1)E$, where $E$ is the exceptional divisor of the blowup. So 
\begin{align*}
f^*\OO(\mA {-} |\mA|) 
= \OO(f^*\mA {-} f^*|\mA|) 
&= \OO(\mB {-} |\mB|) \otimes \OO((1{-}i)E) 
\\&= \OO(\mB {-} |\mB|)(i{-}1).
\end{align*}
Since we are dealing with vector bundles, we can apply the projection formula \cite[01E8]{stacks-project} to find 
\[ Rf_*(\ulMO_{\hB}) 
= Rf_*(f^*(\ulMO_{\hA})(1-i)) 
= \ulMO_{\hA} \otimes_{\OO_{\hA}} Rf_*\OO_{\hB}(1-i). \]
So the result follows from the calculation Proposition~\ref{prop:buCohLine} of $Rf_*\OO_{\hB}(1-i)$ since we have $1 \leq i \leq n$ and therefore $-1-n < 1-n \leq 1-i \leq 0$.
\end{proof}

\section{Cube invariance of $R\Gamma(-, \protect\ulMO)$} \label{sec:MOcube}

The goal of this subsection is to prove that the cohomology presheaves of modulus global sections satisfy cube invariance. 
First we prepare a general criterion for cohomological cube invariance. 
We will use it in Proposition~\ref{prop:ulMOcube} 
to show cube invariance on nice modulus pairs for $\ulMO$.

\begin{lemma}\label{lem:ci-criterion}
Let $F$ be an additive presheaf on $\ulPSch_k$, and let $\tau \in \{\Zar,\Nis,\et\}$.
Let $\sX = (\hX,\mX)$ be a modulus pair such that $\hX$ is quasi-compact and $F_{\sX}$ is a quasi-coherent sheaf of $\OO$-modules on the small site $\hX_{\tau}$.
Suppose moreover that for any affine open subscheme $\hU = \Spec A \subset \hX$ with $\mU := \mX \cap \hU = \Spec A/(f)$ a principal Cartier divisor, the sequence of $A$-modules
\begin{equation}\label{eq:affine-ci}
0 \to F(A,f) \to F(A[t],f) \oplus F(A[\tfrac{1}{t}],f/t) \to F(A[t,\tfrac{1}{t}],f) \to 0
\end{equation}
is exact. Then, for any $i \in \Z$, the first projection $\sX \boxtimes \bcube \to \sX$ induces an isomorphism of abelian groups
\[
H_\tau^i (\hX , F_{\sX}) \xrightarrow{\sim} H_\tau^i (\hX \times \PP^1, F_{\sX \boxtimes \bcube}).
\]
\end{lemma}

\begin{nota}
The operation $\boxtimes$ is the restriction of $\otimes: \ulMCor_k \times \ulMCor_k \to \ulMCor_k$ to $\ulMSm_k$. That is, for modulus pairs $\sX, \sY$, the total space of $\sX \boxtimes \sY$ is $\hX \times \hY$ and the divisor is $\mX \times \hY + \hX \times \mY$. Note that this almost never represents the Cartesian product in $\ulMSm_k$. 
\end{nota}

\begin{proof}
It suffices to treat the case $\tau=\Zar$ since the \'etale and Zariski cohomology agree if $F$ is a quasi-coherent \'etale sheaf, \cite[Rem.III.3.8]{Mil80}.
By Mayer-Vietoris and induction on the minimal size of a finite affine covering of $\hX$, we are reduced to the case when $\hX$ is affine.
Let $\PP^1 = U_0 \cup U_1$ be the standard open covering. 
Then, for any $i=0,1$ and $j>1$, we have $H_{\Zar}^j (\hX \times U_{i},F_{\sX}) = 0$ since $\hX \times U_i$ is affine and $F_{\sX}$ is quasi-coherent by assumption.
Therefore, the Mayer-Vietoris long exact sequence is simplified as 
\begin{align*}
0 &\to H_{\Zar}^0 (\PP^1_{\hX},F) \to H_{\Zar}^0 (U_{0,\hX},F) \oplus H_{\Zar}^0 (U_{1,\hX},F) \to H_{\Zar}^0 (U_{01,\sX},F) \\
&\to H_{\Zar}^1 (\PP^1_{\hX},F) \to 0,
\end{align*}
where $U_{01} = U_0 \cap U_1$, and $(-)_{\hX} = (-) \times \hX$ (we omit the subscripts of $F$ for the simplicity of notation). 
Therefore, the right exactness of Eq.\eqref{eq:affine-ci} shows \[ H_{\Zar}^1 (\PP^1_{\hX},F_{\PP^1_{\hX}}) = 0 = H_{\Zar}^1 (\hX,F_{\sX})\] and the left exactness of  Eq.\eqref{eq:affine-ci} shows $H_{\Zar}^0 (\PP^1_{\hX},F_{\sX \boxtimes \bcube}) = H_{\Zar}^0 (\hX,F_{\sX})$.
\end{proof}

Now, we move on to the proof of the cube invariance of the cohomology of $\ulMO$.
We start with the following lemma.

\begin{lemm} \label{lemm:redNotOO}
Suppose that $A$ is reduced and $f$ is a nonzero divisor. 
Then we have 
\[ \ulMO(A[t], f) = \ulMO(A[t], ft). \]
\end{lemm}

\begin{proof}
By functoriality, there exists a canonical inclusion 
\[ \ulMO(A[t], f) \subseteq \ulMO(A[t], ft) \]
compatible with the inclusion 
$A[t, f^{-1}] \subseteq A[t, t^{-1}, f^{-1}]$.
We wish to show that this inclusion is sujective. 

Choose $a \in \ulMO(A[t], ft) \subseteq A[t, t^{-1}, f^{-1}]$. To show that $a \in \ulMO(A[t], f)$, by Lemma~\ref{lemm:memCritMO} it suffices to show that
\begin{equation} \label{equa:aftft}
af, (af)^na \in A[t]
\end{equation}
for some $n \geq 0$ (note this will imply that $a \in A[t, f^{-1}]$). We will show that for some $n$ these two elements are in both $A[t, \frac{1}{f}]$ and $t^{-n}A[t]$. Then since $t^{-n}A[t] \cap A[t, \frac{1}{f}] = A[t]$ we obtain Eq.\eqref{equa:aftft}.

Write $a = t^mb$ where $b \in A[t, \frac{1}{f}]$ has non-zero constant term and $m \in \ZZ$. Then we have 
\begin{equation} \label{equa:aftbt}
a(ft) = t^mb(ft) \textrm{ and } (aft)^na = (t^mbft)^n t^mb = t^{{mn+m+n}}b^{n+1}f^n \in A[t]
\end{equation}
for some $n \geq 0$ by Lemma~\ref{lemm:memCritMO}.
Since $b \in A[t, \frac{1}{f}]$ has non-zero constant term and $A$ is a reduced ring, $b^{n+1}$ also has non-zero constant term. So we have {$(m+1)(n+1) -1= mn+m+n\geq 0$, which implies $m+1 \geq 1$} and hence $m \geq 0$, which means $a = t^mb \in A[t, \frac{1}{f}]$. So our first goal, $af, (af)^na \in A[t, \frac{1}{f}]$ is achieved. The second goal, $af, (af)^na \in t^{-n}A[t, \frac{1}{f}]$ follows from Eq.\eqref{equa:aftbt}.
\end{proof}

\begin{prop} \label{prop:ulMOcube}
Suppose that $(A,f)$ is a modulus pair with $A$ reduced. Then 
\begin{equation} \label{equa:ulMOcube}
 0 \to \ulMO(A,f) \to \ulMO(A[t],f) \oplus \ulMO(A[\tfrac{1}{t}],f/t) \to \ulMO(A[t, \tfrac{1}{t}],f) \to 0
\end{equation}
is a short exact sequence. 
Consequently, for any $\sX \in \ulPSm_k$ with $\hX$ reduced and for $\tau \in \{\Zar,\Nis,\et\}$, we have
\[ H^n_{\tau}(\sX, \ulMO) = H^n_{\tau}(\sX {\boxtimes} \bcube, \ulMO). \]
\end{prop}

\begin{proof}
Since $A$ is reduced, by Lem.~\ref{lemm:redNotOO}, we may replace $\ulMO(A[\tfrac{1}{t}],f/t)$ with \allowbreak $\ulMO(A[\tfrac{1}{t}],f)$. Then the sequence Eq.\eqref{equa:ulMOcube} is a subsequence of the exact sequence
\begin{equation} \label{equa:ulMOcubeInt}
 0 \to A[\tfrac{1}{f}] \to A[\tfrac{1}{f}][t] \oplus A[\tfrac{1}{f}][\tfrac{1}{t}]
 \to A[\tfrac{1}{f}][t, \tfrac{1}{t}] \to 0.
\end{equation}
Exactness of Eq.\eqref{equa:ulMOcube} at $\ulMO(A,f)$ follows from left exactness of Eq.\eqref{equa:ulMOcubeInt}.

Let's show exactness of Eq.\eqref{equa:ulMOcube} in the middle. Suppose that we have a cycle $(a, b)$ in the middle of Eq.\eqref{equa:ulMOcube}. By exactness of Eq.\eqref{equa:ulMOcubeInt}, $a = b$ is in $A[\tfrac{1}{f}]$. Moreover, by Lem.\ref{lemm:memCritMO} it satisfies $af, (af)^na \in A[t] \cap A[\tfrac{1}{t}] = A$ for some $n \geq 0$. Hence, it comes from an element of $\ulMO(A, f)$.

Now let's show right exactness of Eq.\eqref{equa:ulMOcube}. Suppose that it is not surjective. Choose an element $a = \sum_{i = m}^\ell a_it^i \in \ulMO(A[t, \tfrac{1}{t}], f)$ not in the image. If $\ell \leq 0$ or $m \geq 0$, then this element is in the image (because it satisfies $af, (af)^na \in A[t, \tfrac{1}{t}]$ for some $n$) so we must have $m < 0$ and $\ell > 0$. 
We prove that the condition $\ell > 0$ leads to a contradiction as follows.
Suppose that we have chosen an element such that $\ell$ is minimal. 
The highest degree term of $(af)^na$ is $(a_\ell t^\ell f)^{n} a_\ell t^\ell$. 
But then from $af, (af)^na \in A[t, \tfrac{1}{t}]$ we deduce that $a_\ell t^\ell f$ and $(a_\ell t^\ell f)^{n} a_\ell t^\ell$ are in $A[t]$, so $a_\ell t^\ell$ is in the image of $\ulMO(A[t], f)$. Since $a$ is not in the image, $a - a_\ell t^\ell$ is also not in the image, so $a$ did not have minimal $\ell$; a contradiction.

Finally, the second assertion in the statement follows from Lem.~\ref{lem:ci-criterion} since $\ulMO|_{\hX_{\et}}$ is quasi-coherent \'etale sheaf by Prop.~\ref{prop:fpqc}.
\end{proof}

\begin{rema}
The above exactness is false if $A$ is not reduced, as one sees immediately from the example $(A, f) = (k[\epsilon]/\epsilon^2, 1)$. Indeed, in this case we have
\begin{align*}
\ulMO(A[t], f) &= A[t] \\
\ulMO(A[\tfrac{1}{t}], f/t) &= A[\tfrac{1}{t}] + \langle \epsilon \rangle t \\
\ulMO(A[t, \tfrac{1}{t}], f) &= A[t, \tfrac{1}{t}] \\
\end{align*}
giving global sections of $A + \langle \epsilon \rangle t$, instead of $A$.
\end{rema}

\section{Transfers on {$\protect\ulMO$}} \label{sec:MOtransfers}

We observe that the structure of presheaf with transfers on $\OO$ recalled in Recollections~\ref{reco:Cor}\eqref{reco:Cor:Omega} induces a structure of presheaf with transfers on $\ulMO$.

\begin{lemma} \label{lemm:MOTransfers}
Let $\sX, \sY \in \ulMSm_k$ with $\hX$ normal, and $\alpha \in \hom_{\ulMCor_k}(\sX, \sY) \subseteq \hom_{\Cor_k}(\iX, \iY)$. Then there exists a unique map $\ulMO(\sY) \to \ulMO(\sX)$ making the following diagram commute:
\[\xymatrix{
\ulMO (\sY) \ar[r] \ar[d] & \ulMO (\sX) \ar[d] \\
\OO (\iY) \ar[r]^{\alpha^*} & \OO (\iX).
}\]
\end{lemma}

\begin{rema}
A modulus pair with non-smooth interior will appear in the proof. One checks directly that 
Lemma~\ref{lemm:affineEtQC}, 
Proposition~\ref{prop:fpqc}, 
Lemma~\ref{lemm:etaleAffineCoh}, 
Theorem~\ref{theo:MOdefi}, 
Lemma~\ref{lemm:memCritMO}, 
Lemma~\ref{lemm:intersectHeightOne}, and 
Proposition~\ref{prop:hdescent} all work verbatim for pairs $(\hX, \mX)$ wth $\hX$ Noetherian normal, and $\mX$ an effective Cartier divisor. In fact these work even more generally than that, cf.Remark~\ref{rema:baby}.
\end{rema}

\begin{proof}
By definition, Rec.\ref{reco:Cor}(3), there is a morphism of modulus pairs $\sW \to \sX$ such that $\hW$ is integral, $\hW \to \hX$ is proper surjective, $\iW \to \iX$ is finite, and the composition $\sW \to \sX \to \sY$ is a finite sum of morphisms of modulus pairs. Normalising, we can assume $\hW$ is integrally closed in $\iW$. As such, the morphism $(\ast)$ in the diagram

\[ \xymatrix{
\OO(\iY) \ar[r] &\OO(\iX) \ar[r]^{\subseteq} & \OO(\iW) \\
\ulMO(\sY) \ar@{-->}[r]^{(\ast\ast)} \ar@/_12pt/@{-->}[rr]_{(\ast)} \ar[u]_{\cup|} &\ulMO(\sX) \ar[r] \ar[u]_{\cup|} & \ulMO(\sW) \ar[u]_{\cup|}
} \]
certainly exists, and is unique by injectivity of $\ulMO(\sW) \subseteq \OO(\iW)$. By Proposition~\ref{prop:hdescent} the square on the right is Cartesian, so the morphism $(\ast\ast)$ also exists and is unique.
\end{proof}

\section{Hodge realisation for $\protect\ulMO$}

We now combine the above to prove our main theorem for $\ulMO$, Theorem~\ref{theo:MORealisation}. The idea is that $R\Gamma_{\Nis}(-, \ulMO)$ can be equipped with transfers, and should be blow up invariant and cube invariant. Sadly, Example~\ref{exam:counterGabber} below shows that $R\Gamma_{\Nis}(-, \ulMO)$ is not blowup invariant without some stricter hypotheses. We use normal crossings. 

\begin{nota}
Write $\ulMSm_k^{\nc} \subseteq \ulMSm_k$ (resp. $\ulMCor_k^{\nc} \subseteq \ulMCor_k$) for the full subcategory of quasi-projective normal crossings modulus pairs. 
\end{nota}

\begin{rema} \label{rema:NCMDM}
In general, we have the following.
\begin{enumerate}
 \item If $\sY \to \sX$ is an abstract admissible blowup, then {$\sY \otimes \bcube \to \sX \otimes \bcube$} is again an abstract admissible blowup.\footnote{Recall that $\sX \otimes \bcube$ has total space $\hX \times \PP^1$ and modulus $\mX \times \PP^1 + \hX \times \{\infty\}$.} So any functor on $\ulMCor_k$ which sends $\sX \otimes \bcube \to \sX$ to an isomorphism, also sends $\sY \otimes \bcube \to \sY$ to an isomorphism.
 
 \item If $\sX \in \ulMCor_k^{\nc}$ then $\sX \otimes \bcube \in \ulMCor_k^{\nc}$.
\end{enumerate}

If $k$ satisfies (RoS), then we also have:
\begin{enumerate}
 \item The inclusions $\ulMSm_k^{\nc} \subseteq \ulMSm_k$ and $\ulMCor_k^{\nc} \subseteq \ulMCor_k$ are equivalences of categories.

 \item Consequently, the canonical comparison functor 
\[
\frac{D(\Shv_{\ulMNis}(\ulMCor_k^{\nc}))}{\biggl \langle \ZZ_{\tr}(\sX \otimes \bcube) \to \ZZ_{\tr}(\sX) : \sX \in \ulMCor_k^{\nc} \biggr \rangle}
\to
\frac{D(\Shv_{\ulMNis}(\ulMCor_k))}{\biggl \langle \ZZ_{\tr}(\sX\otimes \bcube) \to \ZZ_{\tr}(\sX) : \sX \in \ulMCor_k \biggr \rangle}
\]
is an equivalence of categories. 
\end{enumerate}
\end{rema}

\begin{theo} \label{theo:MORealisation}
Suppose that $k$ satisfies (RoS) and (WF), Def.\ref{defi:RosWF}, e.g., $char(k) {=} 0$. 
Then there is a unique object $\RMO \in \ulMDM_k^{\eff}$ such that for $\sX$ with normal crossings we have
\[ \hom_{\ulMDM_k^{\eff}}(M(\sX), \RMO[n]) \cong H^n_{\Zar}(\hX, \sqrt{\sI} \otimes \sI^{-1}). \]
\end{theo}

\begin{proof}
Consider $\ulMO \in \PSh(\ulMSm_k)$ from Proposition~\ref{prop:hdescent}. 
By Lemma~\ref{lemm:MOTransfers} this factors through $\ulMCor_k$. Consider its image $\RMO$ in $D(\Shv_{\ulMNis}(\ulMCor_k))$. For any $\sX \in \ulMCor_k$ we have 
\begin{align}
\hom_{D(\Shv_{\ulMNis}(\ulMCor_k))}(\ZZ_{\tr}(\sX), \RMO[n])
&= 
Ext^n(\ZZ_{\tr}(\sX), \ulMO)
\nonumber
\\
&\stackrel{Eq.\eqref{equa:ExtCoh}}{=} 
H_{\ulMNis}^n(\sX, \ulMO)
\label{eq:homDMNSTcolim}
\\
&\stackrel{Prop.\ref{prop:MZarIscolimZar}}{=} 
\colim_{\sY \to \sX} H_{\Nis}^n(\sY, \ulMO)
\nonumber
\end{align}
where the colimit is over abstract admissible blowups. Consider the case that $\sX \in \ulMSm_k^{\nc}$. By Proposition~\ref{prop:aabAreBu} we can assume all $\hY$ are actual blowups of $\hX$, and by (RoS) we can assume that all $\sY$ are normal crossings. By (WF), such $\sY \to \sX$ are zig zags of abstract admissible blowups $\sV \to \sW$ such that $\hV \to \hW$ is a blowup in a regular centre that has normal crossings with $\mW$. Since $\sX \in \ulMSm_k^{\nc}$ by Proposition~\ref{prop:MOBUinv} (and Theorem~\ref{theo:MOdefi}) the functor $H_{\Nis}^n(-, \ulMO)$ sends such $\sV \to \sW$ to isomorphisms. So 
\[ \textrm{Eq.\eqref{eq:homDMNSTcolim}} 
 \cong H_{\Nis}^n(\sX, \ulMO); \qquad \textrm{ for normal crossings } \sX. \]
Finally, by Proposition~\ref{prop:ulMOcube} we deduce that Eq.\eqref{eq:homDMNSTcolim} is cube invariant, at least for normal crossings $\sX$. But this is sufficient to deduce that it is cube invariant for all $\sX \in \ulMSm_k$, Rem.\ref{rema:NCMDM}. So $\RMO$ lies in the full subcategory $\ulMDM_k^{\eff} \subseteq D(\Shv_{\ulMNis}(\ulMCor_k))$.
\end{proof}

\section{Post-script}

Here we collect some odds and ends.

Here is a proof of the claim in Example~\ref{exam:ulMOsncd}.

\begin{lemm}\label{lemm:flatMO}
Let $A$ be a UFD and $f$ a non-zero divisor. Then the $A$-submodule $\ulMO(A) \subset A[f^{-1}]$ is free of rank one. In particular, $\ulMO(A)$ is a flat $A$-module.
\end{lemm}

\begin{proof}
It suffices to show that $\ulMO(A,f) = \{a/f \in A[f^{-1}] : a \in \sqrt{(f)}\}$ is a free $A$-module of rank one.
Since $A$ is a UFD, there exists a decomposition $f = p_1^{m_1}\cdots p_n^{m_n}$, where $p_i$ are irreducible elements in $A$ and $m_i >0$ for all $i=1,\dots,n$. 
Then we have $\sqrt{(f)} = (p_1\cdots p_n) \subset A$, and hence
\[
\ulMO(A,f) = \frac{1}{p_1^{m_1-1}\cdots p_n^{m_n-1}} \cdot A \subset A{[f^{-1}]}.
\]
Since $p_1\cdots p_m$ is a non-zerodivisor as a factor of $f$, we conclude that $\ulMO(A,f) = (p_1^{m_1-1}\cdots p_n^{m_n-1})^{-1}$ is an invertible sheaf on $\hX$.
\end{proof}

Here is a remark about the more general setting.

\begin{rema} \label{rema:baby}
We have already remarked that 
Lemma~\ref{lemm:affineEtQC}, 
Proposition~\ref{prop:fpqc}, 
Lemma~\ref{lemm:etaleAffineCoh}, 
Theorem~\ref{theo:MOdefi},
and
Lemma~\ref{lemm:memCritMO} 
work for general modulus pairs (i.e., $\hX$ a qcqs scheme and $\mX$ an effective Cartier divisor). 

Lemma~\ref{lemm:intersectHeightOne} is a kind of valuative criterion for global sections. There is a much more general version of this lemma. The more general version is valid for any ring $A$ equipped with a nonzero divisor $f$, and the local rings $A_\p$ are replaced with local rings of the relative Riemann-Zariski space, denoted $Val_{\Spec A[f^{-1}]}(\Spec(A))$ in Temkin's article, \cite{Tem11}.  As such Proposition~\ref{prop:hdescent} also holds for general modulus pairs, but with the restriction that $\OO_{\hX}$ be integrally closed in $j_*\OO_{\iX}$, where $j:\iX \subseteq \hX$ is the inclusion.


For Lemma~\ref{lemm:redNotOO}, Proposition~\ref{prop:ulMOcube} one must further assume that the total space is reduced in the general statements.
\end{rema}

Here is a counter-example showing that in Proposition~\ref{prop:MOBUinv} we need to at least assume that $\hX$ has rational singularities inside the divisor.

\begin{exam}[Gabber] \label{exam:counterGabber}
Suppose $f: \hY \to \hX$ is a resolution of singularities of some $\hX$ with non-rational singularities. That is, such that $R^if_*\OO_{\hY} \neq 0$ for some $i > 0$. Suppose that $\hX$ admits a reduced effective Cartier divisor $\mX$ containing the singularities such that $\mY = f^*\mX$ is also reduced. Then $\ulMO_{(\hY, \mY)} = \OO_{\hY}$ and $\ulMO_{(\hX, \mX)} = \OO_{\hX}$ so by assumption $ \ulMO_{\sX} \to Rf_*\ulMO_{\sY}$ is not an isomorphism. 

An explicit example can be produced by considering the affine cone over an elliptic curve (for example): Suppose that $E \subseteq \PP^2_k$ is a smooth curve with $H^1(E, \OO_E) \neq 0$, e.g., an elliptic curve, and choose a $\PP^1_k \subseteq \PP^2_k$ such that $\PP^1_k \cap E$ is reduced. Let $CE \subseteq \AA^3_k$ be the affine cone over $E$ (we recall a construction below). Let $BE = Bl_{CE} \{0\} \to CE$ be the blowup of the singular point of $CE$. Equip $CE$ and $BE$ with the pullbacks $\mC$, $\mB$ of the effective Cartier divisor $\AA^2_k \subseteq \AA^3_k$ corresponding to the $\PP^1_k \subseteq \PP^2_k$ chosen above, and set $\sB = (BE, \mB)$, $\sC = (CE, \mC)$. We claim that 
\begin{equation} \label{equa:BECEiso}
H^1_{\Zar}(BE, \ulMO_{\sB}) \supseteq H_{\Zar}^1(E, \OO_E) \neq 0
\end{equation}
and therefore
\[ \ulMO_{\sC} \to Rf_*\ulMO_{\sB} \]
is not an isomorphism where $f: BE \to CE$ is the canonical morphism. Note $CE$ is affine, so $H^1_{\Zar}(BE, \ulMO_{\sB})$ is precisely the space of global sections of the quasi-coherent sheaf $R^1f_*\ulMO_{\sB}$.

We recall a construction of $BE, CE$. To begin with, recall that the blowup $Bl_{\AA^3}\{0\}$ of $\AA^3$ in the origin is canonically identified with the total space of the line bundle $\OO_{\PP^2}(1)$ on $\PP^2$ via a retraction $\pi: Bl_{\AA^3}\{0\} \to \PP^2$ to the exceptional divisor $\PP^2 \subseteq Bl_{\AA^3}\{0\}$.\footnote{Indeed, classically $Bl_{\AA^3} \{0\}$ is the variety of pairs $(L, x)$ such that $L \subseteq \AA^3$ is a line through the origin and $x \in L$. The projection $Bl_{\AA^3} \{0\} \to \AA^3$ sends $(L, x)$ to $x$, and the retraction $Bl_{\AA^3} \{0\} \to \PP^2$ sends $(L, x)$ to $L$. The exceptional divisor $\PP^2 \subseteq Bl_{\AA^3} \{0\}$ is the set $\{(L, x)\ |\ x = 0\}$.
} Then one can define $BE \subseteq Bl_{\AA^3} \{0\}$ and $CE \subseteq \AA^3$ by forming the Cartesian square on the left and the surjection $f$. We also have the further Cartesian square on the right coming from the inclusion of the exceptional divisor $\PP^2 \subseteq Bl_{\AA^3}\{0\}$.
\[ \xymatrix{
E \ar[d]_\iota & \ar[l]_{\theta} BE \ar[d]_\phi \ar[r]^f & CE \ar[d] && BE \ar[d]_\phi & \ar[l] \ar[d]^\iota E \\
\PP^2 & \ar[l]^\pi \ar[r]_g Bl_{\AA^3}\{0\} & \AA^3 && Bl_{\AA^3} \{0\} & \ar[l] \PP^2
} \]
Note $\PP^2 \cup \pi^{-1}\PP^1 = g^{-1}\AA^2 \subseteq Bl_{\AA^3} \{0\}$, so $E \cup \theta^{-1}(\PP^1 \cap E) = \mB = f^{-1}\mC$.

For the inclusion of Eq.\ref{equa:BECEiso}, first note that since $\PP^1_k \cap E$ is reduced, the effective Cartier divisor $\mB$ is reduced. Indeed, $\pi$ and therefore $\theta$ is an $\AA^1$-bundle. So
\begin{equation} \label{equa:MOOB}
\ulMO_{\sB} \cong \OO_{BE}.
\end{equation}
Since affine schemes have no higher coherent cohomology, we have $R^j\theta_*\OO_{BE} = 0$ for $j > 0$ and so the spectral sequence $H^i_{\Zar}(E, R^j\theta_*\OO_{BE}) \implies H^{i+j}_{\Zar}(BE, \OO_{BE})$ gives an isomorphism 
\begin{equation} \label{equa:BEE}
H^1_{\Zar}(BE, \OO_{BE}) \cong H^1_{\Zar}(E, \theta_*\OO_{BE}).
\end{equation}
Then by definition\footnote{Indeed, we have identified $Bl_{\AA^3} \{0\}$ with the total space of $\OO_{\PP^2}(1)$, i.e., with $\uSpec \oplus_{i \geq 0} \OO_{\PP^2}(i)$, and $BE$ is the fibre product.} $BE = \uSpec \oplus_{i \geq 0} \OO_E(i)$ which contains the direct summand $\OO_E$:
\begin{equation} \label{equa:OOEOOBE}
\theta_*\OO_{BE} \cong \OO_E \oplus \sF 
\end{equation}
Combining Eq.\eqref{equa:MOOB}, Eq.\eqref{equa:BEE}, and Eq.\eqref{equa:OOEOOBE} gives Eq.\eqref{equa:BECEiso}.
\end{exam}

\begin{rema}
Note that $CE$ in Example~\ref{exam:counterGabber} is normal, since it is regular in codimension one, and complete intersection. The singularity is contained inside the divisor $\mC$. Of course, $BE$ is also normal, since it's an affine bundle over a smooth curve, and therefore also smooth.
\end{rema}

\appendix

\section{Resolution of singularities and weak factorisations}\label{appendix-a}

\begin{defi} \label{defi:NC}
Let $\sX$ be a modulus pair and $Z \subseteq \hX$ a closed subscheme. 
We will say that \emph{$Z$ has strict normal crossings with $\mX$} if for every point $x \in \hX$ the local ring $\OO_{\hX, x}$ is regular, and there exists a regular system of parameters\footnote{Cf.\cite[00KU]{stacks-project}.} $t_1, ..., t_n \in \OO_{\hX, x}$ such that
\[
\mX|_{\Spec(\OO_{\hX, x})} = \prod_{a \in A} t_a^{r_a}, \qquad \textrm{ and } \qquad Z|_{\Spec(\OO_{\hX, x})} = \Spec(\mathcal{O}_{\hX,x} / \langle t_b : b \in B \rangle)
\]
for some $r_a > 0$ and $A, B \subseteq \{1, ..., n\}$.

We will say that \emph{$Z$ has normal crossings with $\mX$} if there exists an \'etale covering $\hV \to \hX$ such that $Z \times_{\hX} \hV$ has strict normal crossings with $\mV$.

We say that $\sX$ is a \emph{normal crossings modulus pair} if $\varnothing$ has normal crossings with $\mX$.
\end{defi}

\begin{rema}
Note, $A \cap B \neq \varnothing$ is allowed; in particular, $Z \subseteq \mX$ is allowed. 
\end{rema}

\begin{defi}[{cf.\cite[\S 1.2]{AT19}}]\label{def-wf}
Suppose that $f: \sY \to \sX$ is a abstract admissible blowup between normal crossings modulus pairs, such that $\hY \to \hX$ is an actual blowup of noetherian qe regular schemes. 
A \emph{weak factorisation} of $\sY \to \sX$ is a factorisation
of $f$ in $\ulMSCH$\footnote{Cf.\cite[Def.1.23]{KM21}.}
\[ \sY = \sV_0 \stackrel{s_1}{\cong} \sV_1 \stackrel{s_2}{\cong} \sV_2 \cong \dots \stackrel{s_l}{\cong} \sV_l = \sX \]
such that for each $i = 1, \dots, l$, either $s_i$ or $s_i^{-1}$ is an abstract admissible blowup in $\ulPSCH$ whose total space $\sV_{i-1} \to \sV_i$ (resp. $\sV_{i-1} \leftarrow \sV_i$) is the blowup of a regular closed subscheme which has normal crossings with $\mV_i$ (resp. $\mV_{i-1}$).
\end{defi}

\begin{defi} \label{defi:RosWF}
Consider the following properties that a field $k$ might satisfy.
\begin{enumerate}
 \item[(RoS)] For every $\sX \in \ulPSm_k$, there exists an abstract admissible blowup $\sY \to \sX$ such that $\sY$ is normal crossings and $\hY \to \hX$ is an actual blowup.
 
 \item[(WF)] Every abstract admissible blowup $f: \sY \to \sX$ in $\ulPSm_k$ such that $\hY, \hX$ are smooth and $\hY \to \hX$ is an actual blowup, admits a weak factorisation.
\end{enumerate}
\end{defi}

\begin{theo}[Resolution of Singularities, {\cite[Thm.1.1]{Tem08}, \cite{Hir64}}] \label{theo:ROS}
Let $X$ be a Noetherian quasi-excellent integral scheme of characteristic zero. Then $X$ admits a semi-strict embedded resolution of singularities. In particular, for every closed subscheme $Z \subseteq X$ there is a blowup $f: X' \to X$ with centre disjoint from the regular locus of $X$, such that $X'$ is regular, $Z \times_X X'$ is a normal crossings divisor. 
\end{theo}

\begin{theo}[Weak Factorisation, {cf.\cite[Thm.1.2.1]{AT19}}] \label{theo:WF}
If $k$ is characteristic zero then $k$ satisfies (RoS).
\end{theo}

Notice that (RoS) and (WF) deal with actual blowups. To can turn abstract admissible blowups into actual blowups we using the following.

\begin{prop}[Temkin] \label{prop:aabAreBu}
Suppose that $\sY \to \sX$ is an abstract admissible blowup in $\ulPSm_k$ with $\hX$ quasi-projective (e.g., affine) and integral. Then there exists a second abstract admissible blowup $\sY' \to \sY$ such that $\hY' \to \hX$ is an actual blowup.
\end{prop}

\begin{proof}
This is essential \cite[Cor.3.4.8]{Tem11} which says that any $\iX$-modification $\hY \to \hX$ is dominated by an $\iX$-blow up $\hY' \to \hX$. Here $\iX$-modification is a factorisation $\iX \to \hY \to \hX$ into a schematically dominant morphism $\iX \to \hY$ and a proper morphism $\hY \stackrel{f}{\to} \hX$, and an $\iX$-blow up is an $\iX$-modification $\hY' \to \hX$ such that there exists an $f$-ample $\OO_{\hY'}$-module $\sL$ equipped with a global section which is invertible on $\iX$. The existence of the ample sheaf $\sL$ implies that $f$ is projective, \cite[0B45]{stacks-project}, and therefore it is an actual blow up, \cite[8.1.24]{Liu02}.
\end{proof}

\section{Comparison of $\protect \tau$ and $\protect \ulM\tau$-cohomologies.} \label{sec:MtauCoh}

Our goal in this subsection is Prop.\ref{prop:MZarIscolimZar} which says that the $\ulM\tau$-cohomology is the colimit over abstract admissible blowups of that $\tau$-cohomology for $\tau = \Zar, \et$. 

The same proof works for $\fppf$, but we have been diligently avoiding $\ulMSch_k$ in this article.

\begin{reco}[Small sites] \label{reco3}\ 
\begin{enumerate}
 \item The small Zariski 
 site $\sX_{\Zar}$
 on a modulus pair is the full subcategory of $(\ulPSm_k)_{/\sX}$ whose objects are minimal morphisms $(\hU, \mU) \to (\hX, \mX)$ such that $\hU \to \hX$ is an open immersion 
i.e., \emph{minimal open immersions}. 
It is canonically equivalent to the small Zariski 
site $\hX_{\Zar}$ 
of the total space $\hX$.

 \item The small $\ulMZar$ 
 site $\sX_{\ulMZar}$ 
 on a modulus pair is the essential image of $\sX_{\Zar}$. 
 It follows from \cite[Lem.1.32, Thm.2.13]{KM21} that all objects of $\sX_{\ulMZar}$ 
 are of the form $$\sV \stackrel{t^{-1}}{\to} \sV' \stackrel{f}{\to} \sX' \stackrel{s}{\to} \sX$$ where $s, t$ are (the images of) abstract admissible blowups in $\ulPSm$ and $f$ is the image of a minimal open immersion. 
 Similarly, every morphism in 
$\sX_{\ulMZar}$ 
is also of this form. 
 \item The small sites $\sX_{\Nis}$, $\sX_{\et}$, $\sX_{\ulMNis}$, $\sX_{\ulMet}$ are defined analogously.
\end{enumerate}
\end{reco}

\begin{lemma}\label{lem:finlimmod}
Let $\tau \in \{\Zar,\Nis,\et\}$. Then, for any modulus pair $\sX$ over $k$, the functor 
\[
\alpha : \hX_{\tau} \to \sX_{\ulM\tau}; \quad \hU \mapsto (\hU, \mX \times_{\hX} \hU)
\]
preserves finite limits. 
\end{lemma}

\begin{proof}
Since $\alpha$ send the terminal object to the terminal object, it suffices to show that it preserves fiber products.
This follows from the fact that for any diagram of ambient minimal morphisms $\sB \to \sA \leftarrow \sC$, the modulus pair 
\[
(\hB \times_{\hA} \hC , \text{the pullback of $\mA$})
\]
represents the fiber product $\sB \times_{\sA} \sC$ in $\ulMSm_k$, \cite[Prop.1.33]{KM21}. (Note that this does not hold for non-minimal morphisms).
\end{proof}

\begin{cor}\label{cor:coskmod}
The functor $\alpha$ in Lemma~\ref{lem:finlimmod} commutes with the skeleton functors and the coskeleton functors.\footnote{The skeleton functor is the forgetful functor $\sk: \PSh(\Delta, C) \to \PSh(\Delta_{\leq n}, C)$ and the coskeleton functor is its right adjoint.} More precisely, for any simplicial object $K$ in $\hX_\tau$ and for any $n \geq 0$, we have an isomorphism $\sk_n(\alpha(K)) \cong \alpha(\sk_n K)$ and $\cosk_n(\alpha(K)) \cong \alpha(\cosk_n K)$.
\end{cor}

\begin{proof}
Since the skeleton functor is induced by the restriction by the inclusion $\Delta_{\leq n} \subset \Delta$, we obviously have $\sk_n(
\alpha(K)) = 
\alpha(\sk_n K)$. 
To show the other assertion, by Lemma~\ref{lem:finlimmod} it suffices to note that $\cosk_n$ is constructed by finite limits (see \cite[0183]{stacks-project}). 
\end{proof}

\begin{lemma}\label{lem:Mtauhyp-criterion}
Let $\tau \in \{\Zar,\et\}$.
Let $\sX$ be a modulus pair over $k$, and let $\hU_\bullet$ be an $n$-truncated simplicial object in $\hX_{\tau}$ for some $n \geq 0$.
Define an $n$-truncated simplicial object $\sU_\bullet$ in $\sX_{\ulM\tau}$ by $\sU_i := (\hU_i,\mX \times_{\hX} \hU_i)$.
Then $\hU_\bullet$ is an $n$-truncated $\tau$-hypercovering if and only if $\sU_\bullet$ is an $n$-truncated $\ulM\tau$-hypercovering.
\end{lemma}

\begin{proof}
First we treat the case $n=0$. 
If $\hU_0 \to \hX$ is a $\tau$-covering, then $\sU_0 \to \sX$ is an $\ulM\tau$-covering by definition of $\ulM\tau$. 
Conversely, suppose that $\sU_0 \to \sX$ is an $\ulM\tau$-covering. 
Since $\hU_0 \to \hX$ is \'etale (resp. locally an open immersion) as an object of $\hX_\et$ (resp. $\hX_\Zar$), it suffices to show that it is surjective. 
By \cite[Cor.~4.21]{KM21}, the associated morphism $\sU_0 \to \sX$ in $\ulMSm_k$ is refined by a composition of minimal ambient morphisms
\[
\sV \xrightarrow{f} \sX' \xrightarrow{s} \sX,
\]
where $s$ is an abstract admissible blow-up, and $f$ is a $\tau$-covering. So we have the solid commutative square
\[ \xymatrix@!=3pt{
& \ar@{-->}[dr]^g \ar@{-->}[dl]_t & \\
\sV \ar[d]_f \ar[rr]^\phi && \sU_0 \ar[d] \\
\sX' \ar[rr]_s && \sX
} \]
in $\ulMSm_k$. As we observed in Recollection~\ref{reco:1} the morphism $\phi$ can be written as a composition $\phi = g \circ t^{-1}$ for some abstract admissible blowup $t$ and some minimal morphism $t$, giving the dashed morphisms making a commutative triangle. Since $t$, $f$ and $s$ are surjective on the total spaces, so is $\sU_0 \to \sX$.

Next we treat the case $n > 0$. For any $m < n$, consider the canonical morphisms 
\begin{equation}\label{eq:c}
c : \hU_{m+1} \to (\cosk_m \sk_m \hU_\bullet)_{m+1}
\end{equation}
and 
\[
d: \sU_{m+1} \to (\cosk_m \sk_m \sU_\bullet)_{m+1},
\]
where $c$ is a morphism of schemes and $d$ is a morphism in $\ulMSm_k$.
Since we know that $\hU_0 \to \hX$ is a $\tau$-covering if and only if $\sU_0 \to \sX$ is an $\ulM\tau$-covering by the base case $n=0$, it remains to show that $c$ is a $\tau$-covering if and only if $d$ is an $\ulM\tau$-covering.
But by Cor.\ref{cor:coskmod}, we may assume that the underlying scheme of $\cosk_m \sk_m \sU_\bullet$ is given by $\cosk_m \sk_m \hU_\bullet$, and hence that $d$ is represented by $c$. 
Then the desired assertion follows from the base case $n=0$.
\end{proof}

\begin{lemma}\label{lem:findiagmodel}
Let $\tau \in \{\Zar,\et\}$, and $\sX$ a modulus pair over $k$. Then, for any finite diagram $\sU_\bullet :I \to \sX_{\ulM\tau}$, there exist an abstract admissible blow-up $\sX' \to \sX$ and a finite diagram $\hV_\bullet :I \to \hX'_{\tau}$ such that $\sV_\bullet \to \sX' \to \sX$ is isomorphic to $\sU_\bullet \to \sX$, where $\sV_i := (\hV_i,\mX \times_{\hX} \hV_i)$ for each $i \in I$.
\end{lemma}

\begin{proof}
We discuss the \'etale case, but the same argument works for the Zariski case. As we observed in Recollection~\ref{reco3}, every object of $\sX_{\ulMet}$ is of the form $\sV \stackrel{t^{-1}}{\to} \sV' \stackrel{f}{\to} \sX' \stackrel{s}{\to} \sX$ where $s, t$ are the images in $\ulMSm_k$ of abstract admissible blowups, and $\sV' \in \sX'_{\et}$. So up to replacing $\sU_\bullet$ with an isomorphic diagram, and replacing $\sX$ with a sufficiently large abstact admissible blowup, we can assume that all $\sU_i$ are in the strict image of $\sX_{\et} \to \sX_{\ulMet}$ (not just in the essential image).

Again applying Recollection~\ref{reco3}, for every $\phi: i \to j$ in $I$, we can write $\sU_\phi: \sU_i \to \sU_j$ as 
\[ \sU_i \stackrel{t_\phi^{-1}}{\to} \sW_\phi \stackrel{s_\phi f_\phi}{\to} \sU_j \]
with $f_\phi$, $s_\phi$, $t_\phi$ as above. This gives a new diagram indexed by the barycentric subdivision $sd(I)$ of the directed graph $I$. Here, $sd(I)$ is the directed graph which has a span $i \stackrel{\sigma_\phi}{\leftarrow} \phi \stackrel{\psi_\phi}{\to} j$ for every edge $i \stackrel{\phi}{\to} j$ of $I$.\footnote{ More explicitly, $sd(I)$ has set of vertices the disjoint union $V_{sd(I)} = V_I \sqcup E_I$ of the vertices and edges of $I$, and set of edges $E_{sd(I)} = E_I \sqcup E_I$ two copies of $E_I$. The source morphism $E_{sd(I)} \to V_{sd(I)}$ is the identity on both copies of $E_I$. The target is the sum of the source and target morphisms $E_I \rightrightarrows V_I$ of $I$.} %
By construction, this new diagram factors as $sd(I) \stackrel{\sV}{\to} \ulPSm \to \ulMSm$. Now consider the disjoint unions $\sW = \sqcup_{Ar(I)} \sW_\phi$ and $\sU = \sqcup_{Ar(I)} \sU_{target(\phi)}$ with the canonical morphisms $\sW \stackrel{t}{\to} \sU \to \sX$ in $\ulPSm$.
By (the proof of) \cite[Thm.2.13]{KM21} there exists an abstract admissible blowup $\sX' \to \sX$ in $\ulPSm$ such that when we form the pullbacks
\[ \xymatrix{
\sX' \times_{\sX} \sW \ar[d]_-{(\ast)} \ar[r] & \sW \ar[d]^t \\
\sX' \times_{\sX} \sU \ar[d] \ar[r]^-{(\ast\ast)} & \sU \ar[d] \\
\sX' \ar[r] & \sX
} \]
in $\ulPSm$, the morphism $(\ast)$ becomes an isomorphism in $\ulPSm$. It follows that $\sX' \times_\sX \sV_\bullet: sd(I) \to \ulPSm$ is actually indexed by $I$, since all ``backwards'' edges $\sigma_\phi$ of $sd(I)$ are sent to isomorphisms. The horizontal morphisms $(\ast\ast)$ are abstract admissible blowups, so they assemble to give a natural isomorphism from $\sX' \times_\sX \sV_\bullet$ to $\sU_\bullet$ in $\ulMSm$. So now we have a diagram in $\ulPSm_{/\sX}$ whose objects are all in $\sX_{\et}$. Since the inclusion $\sX_{\et} \to \ulPSm_{/\sX}$ is fully faithful, our new diagram factors through $\sX_{\et}$. 
\end{proof}

\begin{cor}\label{cor:truncatedhyp-model}
Let $\tau \in \{\Zar,\et\}$.
Let $\sX$ be a modulus pair over $k$, and let $\sU_\bullet$ be an $n$-truncated $\ulM\tau$-hypercovering of $\sX$ in $\ulMSm_k$ for some $n \geq 0$.
Then there exists an abstract admissible blow-up $\sX' \to \sX$ and an $n$-truncated $\tau$-hypercovering $\hU'_\bullet \to \hX'$ such that the induced morphism of simplicial objects
$\sU'_\bullet \to \sX' \to \sX$ is isomorphic to $\sU_\bullet \to \sX$, where $\sU'_m := (\hU'_m,\mX \times_{\hX} \hU'_m)$ for each $m \leq n$.
\end{cor}

\begin{proof}
By Lem.\ref{lem:findiagmodel}, there exist an abstract admissible blow-up $\sX' \to \sX$ and a simplicial object $\hV_\bullet$ in $\hX'_\tau$ such that $\sV_\bullet \to \sX' \to \sX$ is isomorphic to $\sU_\bullet \to \sX$, where $\sV_\bullet$ is given by $\sV_m := (\hV_m,\mX \times_{\hX} \hV_m)$. 
Since $\sU_\bullet \to \sX$ is an $n$-truncated $\ulM\tau$-hypercovering and since $(\sU_\bullet \to \sX) \cong (\sV_\bullet \to \sX')$, Lem.\ref{lem:Mtauhyp-criterion} shows that $\hV_\bullet \to \hX'$ is an $n$-truncated $\tau$-hypercovering.
\end{proof}

\begin{prop} \label{prop:MZarIscolimZar}
For any modulus pair $\sX$ over $k$ and any presheaf of abelian groups $F \in \PSh(\sX_{\ulM\tau})$, we have
\[ R^n \Gamma_{\ulM\tau}(\sX, F) = \colim_{\sX' \to \sX}  R^n \Gamma_{\tau}(\hX', F|_{\hX'_{\tau}}) \]
for all $n\in \Z$ and $\tau \in \{\Zar, \et\}$, where the colimit is over abstract admissible blowups.
\end{prop}

\begin{proof}
By Verdier's hypercovering theorem, [SGA4, Expose V, Sec.7, Thm.7.4.1], for any category with finite limits $C$ equipped with a finitary\footnote{Finitary means every covering family $\{U_i \to X\}$ admits a finite subfamily which is still a covering family.} topology $\tau$, and additive presheaf of abelian groups $F$ we have
\[ H^n_\tau(C, F) \cong \varinjlim_{U_\bullet \to X} H^n(F(U_\bullet)) \]
where the colimit is the filtered colimit over the category of $m$-truncated hypercoverings of the terminal object $X$ and $m > n$.
In particular, we have
\[
R^n \Gamma_{\ulM\tau}(\sX, F) = \colim_{\sU_\bullet \to \sX} H^n (F(\sU_\bullet)),
\]
and 
\[
\colim_{\sX' \to \sX}R^n \Gamma_{\tau}(\hX', F|_{\hX'_{\tau}}) 
= \colim_{\sX' \to \sX}\colim_{\hU'_\bullet \to \hX'} H^n (F(\sU'_\bullet)),
\]
where 
$\sU_\bullet \to \sX$ runs over $(n+1)$-truncated $\ulM\tau$-hypercoverings of $\sX$, $\hU'_\bullet \to \hX'$ runs over $(n+1)$-truncated  $\tau$-hypercoverings of $\hX'$, and $\sU'_\bullet := (\hU'_i, X^{\prime \infty} \times_{\hX'} \hU'_i)_i$.
Note that there exists a natural morphism 
\[
\colim_{\sX' \to \sX}\colim_{\hU'_\bullet \to \hX'} H^n (F(\sU'_\bullet))
\to
\colim_{\sU_\bullet \to \sX} H^n (F(\sU_\bullet))
\]
since the index category on the left is contained in the one on the right.
It suffices to show the inclusion of the opposite direction, but this is a direct consequence of Corollary~\ref{cor:truncatedhyp-model}. 
\end{proof}

\section{A quasi-coherent cohomology calculation} \label{sec:cohCohCalc}

\begin{prop}[{\cite[Prop.2.1.12]{EGAIII}, \cite[01XT]{stacks-project}}] \label{prop:projCoh}
For any ring $A$, we have
\begin{enumerate}
 \item $H^i(\PP^n_A, \OO(\ast)) = 0$ for $i \neq 0, n$.
 \item The canonical homomorphism of graded rings 
 \[ A[t_0, \dots, t_n] \stackrel{\sim}{\to} H^0(\PP^n_A, \OO(\ast)) \]
 is a bijection.
 \item 
 \[ \tfrac{1}{t_0\dots t_n}A[\tfrac{1}{t_0}, \dots, \tfrac{1}{t_n}] \stackrel{\sim}{\to} H^n(\PP^n_A, \OO(\ast)) \]
where the morphism sends an element on the left to the corresponding section of the \v{C}ech cohomology with respect to the standard covering, and the left hand side has the standard grading. In particular, the highest degree nonzero elements are $\tfrac{a}{t_0\dots t_n}$, for $a \in A \setminus \{0\}$, and these have degree $-n-1$.  
\end{enumerate}
\end{prop}

\begin{prop}[{cf.[SGA6, VII, Lem.3.5]}] \label{prop:buCohLine}
Let $k$ be a ring and write $\AA^n := \AA^n_k$ for all $n \geq 0$.
Let $f: B_{n+1} = Bl_{\AA^{n+1}}\{0\} \to \AA^{n+1}$ be the blowup of affine $(n+1)$-space at the origin, and let $\mathcal{O}(1)$ be the line bundle associated to the exceptional divisor. Set $\mathcal{O}(i) := \mathcal{O}(1)^{\otimes i}$ for all $i \in \Z$.
Then $f_*\OO(i)$ is the coherent sheaf associated to $I^i$ where $I$ is the ideal of the origin, and we set $I^i := \Gamma(\AA^{n+1}, \OO_{\AA^{n+1}})$ for $i < 0$.
Moreover, we have
\[ R^qf_*(\OO(i)) = 0 \]
for all $q > 0$ and $i > {-n-1}$.
\end{prop}

\begin{proof}
The statement about global sections follows from a direct calculation. Indeed, if $\AA^n = \Spec(k[t_0, \dots, t_n])$, then on the $k$th standard open \[U_k := \Spec(k[\tfrac{t_0}{t_k}, \dots, t_k, \dots, \tfrac{t_n}{t_k}]) \subset B_{n+1},\] the line bundle $\OO(i)$ is the free sub-$k$-module of $k[t_0, \dots, t_n, t_0^{-1}, \dots, t_n^{-1}]$ generated by monomials $t_0^{r_0} \dots t_n^{r_n}$ such that $r_j \geq 0$ for $j \neq k$, and $r_k \geq i - \sum_{j \neq k} r_j$. The intersection $\bigcap_{k = 0}^n \OO(i)(U_k)$ of these groups is the free abelian group generated by monomials $t_0^{r_0} \dots t_n^{r_n}$ subject to the condition $r_0, \dots, r_n \geq 0$ if $i \leq 0$, and subject to the further condition $\sum_j r_j \geq i$ if $i \geq 0$. Hence, the claim in the statement.

Next we prove the vanishing assertion. 
Since $R^qf_*(\mathcal{O}(i))$ is a coherent sheaf on $\AA^{n+1}$, it suffices to show that its global section vanishes. 
Consider the short exact sequences
\[ 0 \to \OO_{B_{n+1}}(i+1) \to \OO_{B_{n+1}}(i) \to \phi_*(\OO_{\PP^{n}}(i)) \to 0 \]
where $\phi: \PP^{n} \hookrightarrow B_{n+1}$ is the canonical inclusion of the exceptional divisor. 
\[
\cdots \to R^{q} f_*\OO_{B_{n+1}}(i+1) \to R^{q} f_*\OO_{B_{n+1}}(i) \to R^{q} f_*\phi_*(\OO_{\PP^{n}}(i)) \to \cdots.
\]
Moreover, noting that $\phi$ is an affine morphism, we have \[R^{q} f_*\phi_*(\OO_{\PP^{n}}(i)) = H^q (\PP^n ,\OO_{\PP^{n}}(i)),\] which vanishes when $i > -(n+1)$ by Prop.\ref{prop:projCoh}(3).
By Serre vanishing for proper morphisms (which is valid for any noetherian base, see \cite[Prop.~2.6.1]{EGAIII}, \cite[Lem.~0B5U]{stacks-project}), since $\OO(1)$ is ample, there is some $N$ such that $R^qf_*(\OO_{B_{n+1}}(i)) = 0$ for all $i \geq N$ and $q > 0$.
Therefore, descending induction starting with $i = N$ shows that $R^qf_*\OO_{B_{n+1}}(i) = 0$ for $q > 0$ and $i > {-(n+1)}$.
\end{proof}

\bibliographystyle{halpha-abbrv}
\bibliography{bib}

\printindex
\printindex[not]

\end{document}